\documentclass[12pt]{amsart}
\usepackage{amsmath,graphicx,amssymb,amsthm}

\allowdisplaybreaks[4]

\newcommand{\sett}[1]{\left\{#1\right\}}

\newcommand{\A}{\mathcal{A}}

\def\A{\mathcal A}

\def\amslatex{$\mathcal{A}\kern-.1667em\lower.5ex\hbox{$M$}\kern-.125em\mathcal{S}$-\LaTeX}

\newcommand{\abs}[1]{\left\vert#1\right\vert}
\newcommand{\norm}[1]{\left\Vert#1\right\Vert}
\newcommand{\h}{\mathcal{H}}

\newtheorem{set}{set}[section]
\newtheorem{Corollary}[set]{Corollary}
\newtheorem{Definition}[set]{Definition}

\newtheorem{Lemma}[set]{Lemma}

\newtheorem{Remark}[set]{Remark}
\newtheorem{Theorem}[set]{Theorem}
\newcommand{\define}{\mathrel{\hbox{$\equiv$\hskip -.90em \lower .47ex \hbox{$\leftharpoondown$}}}}
\newcommand{\enifed}{\mathrel{\hbox{$\equiv$\hskip -.90em \lower .47ex \hbox{$\rightharpoondown$}}}}

\begin{document}
\title{Mixing and weakly mixing abelian subalgebras of type $\rm{II}_{1}$ factors}

\author[Cameron]{Jan Cameron}
\address{\hskip-\parindent
Jan Cameron, Department of Mathematics, Vassar College, Poughkeepsie, New York,
U.S.A.}
\email{jacameron@vassar.edu}
\author[Fang]{Junsheng Fang}
\address{\hskip-\parindent
Junsheng Fang, Department of Mathematics, Dalian University of Technology, Dalian, China.}
\email{junshengfang@gmail.com}
\author[Mukherjee]{Kunal Mukherjee}
\address{\hskip-\parindent
Kunal Mukherjee, Department of Mathematics, Indian Institute of Technology Madras, Chennai, India.}
\email{kunal@iitm.ac.in}

\date{15 Aug, 2014}

\maketitle

\begin{abstract}
This paper studies weakly mixing $($singular$)$ and mixing masas in type $\rm{II}_{1}$ factors 
from a bimodule point of view. Several necessary and sufficient conditions to characterize 
the normalizing algebra of a masa are presented. We also study the structure of mixing inclusions, with special attention paid to masas of product class.
A recent result of Jolissaint and Stalder concerning mixing masas arising out of inclusions of groups is 
revisited. One consequence of our structural results rules out the existence of certain Koopman-realizable measures, arising from semidirect products, which are absolutely continuous but not Lebesgue. We also show that 
there exist uncountably many pairwise non--conjugate mixing masas in the free group factors 
each with Puk\'{a}nszky invariant $\{1,\infty\}$.

\end{abstract}

\section{Introduction}

This paper is a continuation of work initiated by the authors in  
\cite{CFM}, and deals with maximal abelian self--adjoint subalgebras $($masas in the sequel$)$ in finite von Neumann algebras.  
In particular, we study the various notions of singular masas in $\rm{II}_{1}$ factors that arise from notions 
in ergodic theory and undertake a systematic analysis of the bimodules associated to these masas. 
While the techniques used in \cite{CFM} were more operator algebraic in nature,
those used in this note are primarily measure-theoretic, so there is little overlap between the techniques deployed here and those found in \cite{CFM}. It will be evident from the results in \S3--\S6 that the measure-theoretic approaches in this paper yield new insights into the notion of singularity in the context of masas: among other results, in what follows we relate the mixing properties of masas to those of certain associated measures; build on a number of recent results on strong singularity in \cite{J-S, Muk1, Muk2,S-S3, SSWW}; and give a new construction of many non-conjugate mixing masas in the free group factors.\\
\indent The study of normalizers of masas in $\rm{II}_{1}$ factors has a long history which 
dates back to 1954 \cite{Dix1}. Let $M$ be a finite von Neumann algebra gifted with a fixed 
faithful, normal, tracial state $\tau$. Let $M$ be in its  standard 
form i.e., acts on the GNS space associated to $\tau$ by left multiplication. 
Given a masa $A\subset M$, Dixmier defined the group of normalizing unitaries of $A$ to be 
$N(A)=\{u\in \mathcal{U}(M):uAu^{*}=A\}$, where $\mathcal{U}(M)$ denotes the unitary group 
of $M$ \cite{Dix1}. He named $A$ to be singular, if $N(A)$ is as small as possible i.e., 
$N(A)\subset A$. This definition is purely algebraic but the true nature of singularity 
was first unveiled by Nielsen \cite{Ni}. 
In \cite{Ni}, it was shown that for a $($free$)$ ergodic measure preserving action $\alpha$ 
of a countable discrete abelian group $G$ on a standard probability space, the copy of the group in the associated crossed product produces a singular masa if and only if $\alpha$ is weakly mixing $($see \cite{J-S, Kre} for the relevant definitions$)$.\\
\indent Note that every separable $\rm{II}_{1}$ factor has singular masas \cite{Po2}. The authors of \cite{RobSS,S-S3} 
studied the notion of  `infinity--two' norm to handle singular masas and defined an apparently stronger notion of 
singularity known as strong singularity. Subsequently, in \cite{SSWW} it was proved  that strong 
singularity is equivalent to singularity and:
\begin{Theorem}\label{thm_si}\emph{\cite{SSWW}}
The masa $A\subset M$ is singular if and only if, 
given any finite set $x_{i}\in M$ with $\mathbb{E}_{A}(x_{i})=0$ for all $i$ and for 
every $\epsilon>0$, there exists a unitary $u\in A$ such that $\norm{\mathbb{E}_{A}(x_{i}ux_{j})}_{2}<\epsilon$ for all $i,j$, where $\mathbb{E}_{A}$ 
denotes the unique trace preserving normal conditional expectation onto $A$.
\end{Theorem} 
\noindent This last property is known as the weak asymptotic homomorphism property 
$($WAHP in the sequel$)$.    It was shown in \cite[Theorem ~6.6]{Muk1} that the unitary in the definition of the WAHP can always be chosen to be $v^k$, for some positive integer $k$, where $v \in A$ is a Haar unitary generator of the masa $A$. \\
\indent 
The main topic of this paper -- the notion of a strongly mixing masa in a finite von Neumann algebra -- was introduced by Jolissaint and Stalder in \cite{J-S} $($see Definition \ref{sm}$)$. 
They proved that if a countable discrete abelian group $G$ acts on $M$ by $($free$)$ mixing $\tau$--preserving automorphisms, then the copy of the group in the associated crossed product produces a strongly mixing masa (in what follows, we will use the terms `strongly mixing masa' and `mixing masa' interchangeably).  In \cite{CFM}, the authors gave a formulation of Jolissaint and Stalder's  definition of strong mixing for general subalgebras, and showed that one can replace the groups of unitaries found in the definition of \cite{J-S} by bounded weakly 
null sequences. The following definition was shown in \cite{CFM} to be equivalent to the one in \cite{J-S}, in the context of masas.
\begin{Definition}$($Compare with \cite[Definition 3.4]{J-S}$)$\label{sm}
A masa $A\subset M$ is strongly mixing, if for any bounded sequence of operators $a_{n}\in A$ converging to zero in the $w.o.t$ $($weakly null in the sequel$)$ 
and $x,y\in M$ such that $\mathbb{E}_{A}(x)=0$ and $\mathbb{E}_{A}(y)=0$, 
one has that $\norm{\mathbb{E}_{A}(xa_{n}y)}_2 \rightarrow 0$ as 
$n\rightarrow\infty$. 
\end{Definition}
\indent  The outline of the paper is as follows.   In Section 2, we will associate to a masa in a II$_1$ factor a measure class on a compact space of the form $X \times X$, called the left-right measure, as well as a multiplicity function on the same space $($see Definition \ref{mminv}$)$; these two objects together encode the structure of the standard Hilbert space as a natural bimodule over the masa.  Most of the subsequent analysis in the paper will focus on the measure class.  In the third section, we present a set of equivalent conditions that 
characterize the operators in the normalizing algebra of a general masa in a II$_1$ factor $($Theorem \ref{tfae}$)$. Crucial to 
the proof of this result is a classical theorem of Wiener on Fourier coefficients of measures. Theorem \ref{CorChanged}, a main outcome of our analysis in Section 3, is a generalization of the asymptotic homomorphism property introduced in \cite{S-S3}. We show  there is a sequence of positive integers $k_{l}$ such that $\norm{\mathbb{E}_{A}(xv^{k_{l}}x^{*})}_{2}\rightarrow 0$ as $l\rightarrow \infty$ for all $x\in M$ such that $\mathbb{E}_{N(A)^{\prime\prime}}(x)=0$, 
where $v$ is a Haar unitary generator of the masa $A$.  These statements are independent of the
choice of coordinates, i.e., the choice of the Haar unitary generator. \\
\indent In Section 4, we study a special class of masas called masas of product class $($also studied in \cite{Muk2}$)$, which possess vigorous mixing properties. In particular, they are mixing, $($Theorem \ref{summability_imply_mixing}$)$, 
and the convergence in the definition of mixing for masas of product class is the stronger notion of almost everywhere convergence
$($Theorem \ref{aeconv}$)$. Our consideration of masas of product class was inspired by a similar, though slightly 
different, class of masas originating in work of Sinclair and Smith $($cf. \cite[\S11]{S-S2}$)$. For masas of product class, 
we demonstrate the existence of sufficiently many vectors $\zeta\in L^{2}(M)\ominus L^{2}(A)$ for which $\mathbb{E}_{A}(\zeta v^{n}\zeta^{*})=0$ for all $n\neq 0$, where $v$ is a Haar unitary generator of $A$ $($Theorem \ref{density_wand_vect}$)$. 
A slightly weaker form of the previous statement is a special property for masas in free group factors by a deep theorem of Voiculescu \cite{Voi}. We also show that a large class of 
mixing masas in $\rm{II}_{1}$ factors, namely those which arise out of inclusions of groups, fall 
in the product class $($Theorem \ref{inclusion_of_subgroups}$)$. This invigorates one of 
the main questions addressed in \cite[Theorem 3.5]{J-S}. This result seems to be very 
interesting,  as the combinatorial relations in \cite{J-S} that determine mixing turn out to be  spectral 
analytic in nature. \\
\indent The subject of Section 5 is mixing masas that arise from measurable dynamical systems.  A Radon measure $\mu$ on $[0,1]$ is said to be mixing (or, sometimes, Rajchman) if its
Fourier coefficients $\widehat{\mu}_{n}=\int_{0}^{1}e^{2\pi
int}d\mu(t)$ converge to zero as $\abs{n}\rightarrow\infty$. One
can also define mixing measures on the circle $($more generally on separable 
compact abelian groups$)$ by integrating the
functions $z^{n}$ with respect to the measure. The measure $\mu$ is called weak 
mixing or non--atomic $($see \cite{Kat}$)$, if its Fourier coefficients converge 
strongly $($in absolute value$)$ to zero in the sense of Ces\`{a}ro.  
We show that for masas arising out of mixing 
actions along the direction of the groups in the associated crossed products, 
the disintegrations of their left--right measures with respect to the coordinate projections yield mixing measures for almost all 
fibres $($Theorem \ref{action_sm}$)$. This, in turn, justifies the terminology `mixing masa.'
We compute the left-right measures of masas arising out of group actions by relating it to the 
maximal spectral types of the actions $($Theorem \ref{left--right}$)$. This was stated 
in \cite{Ne-St}, but our way of calculating has some advantages; 
the operators and expressions involved in the definition of WAHP and mixing naturally pop up in our
calculation. For standard results 
about direct integrals used in these analyses, we refer the reader to \cite{KadR}.\\
\indent  The following is an old open problem in ergodic theory $($see Remark \ref{DynamNoexist}$)$: 
Can the maximal spectral type of an ergodic transformation be absolutely continuous but not Lebesgue? For an excellent 
account on these class of problems check \cite[\S6]{KatL}. Measures $($strictly speaking 
equivalence classes of measures$)$ which arise as maximal spectral type of $\mathbb{Z}$--systems will be called Koopman--realizable. When Theorem \ref{inclusion_of_subgroups} is combined with Theorem \ref{left--right}, we conclude:  There does not exist any 
countable discrete non abelian group of the form $G\rtimes \mathbb{Z}$ such that\\ 
$(i)$ $L(\mathbb{Z})\subset L(G\rtimes \mathbb{Z})$ is a mixing masa,\\
$(ii)$ the maximal spectral type of the $\mathbb{Z}$--action on $L(G)$ is absolutely continuous but 
not Lebesgue $($see Corollary \ref{notexist}, Corollary \ref{formal}$)$.\\
\indent Borrowing ideas from ergodic theory, in Section 6 we exhibit examples of uncountably many pairwise 
non conjugate $($by automorphisms$)$ mixing masas in the free group factors each with Puk\'{a}nszky invariant 
$\{1,\infty\}$ $($Theorem \ref{arbitrary}$)$.  Finally, a technical result concerning left-right measures, that is necessary in the later sections
of the paper, is proved in the appendix.\\

\noindent \textbf{Acknowledgements}: JC was partially supported by a research travel grant from the Simons Foundation, and by Simons Collaboration Grant for Mathematicians \#319001.  The authors thank the anonymous referee for many suggestions for improving the presentation of the paper.  

\section{Preliminaries and Setup}

\indent All measure spaces appearing in this paper are assumed to be standard Borel spaces, 
all von Neumann algebras are assumed to be separable, and all inclusions of algebras are unital. 
Whereas in \cite{CFM}, we  dealt with general inclusions of finite von Neumann algebras, in this paper, 
we focus on the special case of masas in $\rm{II}_{1}$ factors. Let $M$ be a $\rm{II}_{1}$ factor, equipped 
with a faithful,
normal, tracial state $\tau$. The trace $\tau$ induces an inner product on 
$M$ by $\langle x,y\rangle=\tau(y^{*}x)$, $x,y\in M$, and thus induces a 
Hilbert norm $\norm{\cdot}_{2}$ on $M$. Let $L^2(M):=L^{2}(M,\tau)$ denote 
the completion of $M$ in $\norm{\cdot}_{2}$. We assume that $M$ acts on
$L^2(M)$ via left multiplication. Let $J$ denote the 
conjugation operator on $L^{2}(M)$, obtained from extending the densely
defined map $J:M\subset L^{2}(M)\rightarrow M\subset L^{2}(M)$ by $Jx=x^{*}$. 
The image of a $L^{2}$--vector
$\zeta$ under $J$ will be denoted by $\zeta^{*}$. Let $A\subset M$ be a
masa and let $e_{A}:L^{2}(M)\rightarrow L^{2}(A)$ be the Jones
projection associated to $A$, where $L^{2}(A)=\overline{A}^{\norm{\cdot}_{2}}$. 
Denote $\A=(A\cup
JAJ)^{\prime\prime}$. It is well known that $e_{A}\in \A$ $($cf. \cite{S-S2}$)$. 
The one--norm on $M$ is defined as
$\norm{x}_{1}=\tau(\abs{x})$, $x\in M$. It is also true that 
$\norm{x}_{1}=\underset{y\in M:\norm{y}\leq 1}\sup\abs{\tau(xy)}$, $x\in M$. 
The completion of $M$ in $\norm{\cdot}_{1}$ is denoted by $L^{1}(M)$. 
The $\tau$--preserving normal conditional expectation $\mathbb{E}_{A}$ onto $A$, 
the trace and $J$ extend to
$L^{1}(M)$ in a continuous fashion, and, $L^{1}(M)$ is also identified with 
the predual of $M$. With abuse of notation, we write 
$e_{A}(\zeta)=\mathbb{E}_{A}(\zeta)$ for $L^{1}(M)$ and $L^{2}(M)$
vectors. Similarly, we will use the same symbols $\tau$ and $J$ to denote their
extensions. This will be clear from the context and will cause no
confusion. For more details check \cite{S-S2}.\\
\indent Given a type $\rm{I}$ von Neumann algebra $B$,
write Type($B$) for the set of all those
$n\in \mathbb{N}\cup \{\infty\}$, such that $B$ has a nonzero component of
type $\rm{I}_{n}$. The Puk\'{a}nszky invariant of a masa $A\subset
M$, denoted by $Puk(A)$ (or $Puk_{M}(A)$ when the containing factor
is ambiguous) is defined to be Type($\A^{\prime}(1 - e_{A}))$ \cite{Pu}.
A stronger invariant for masas in $\rm{II}_1$ factors called the measure--multiplicity invariant  was studied in \cite{DSS,Muk1,Muk2,Ne-St}, and was used in
\cite{DSS,Muk2} to distinguish masas with same Puk\'{a}nszky
invariant \cite{Pu}.  The measure--multiplicity invariant of a masa has two main components: a measure class
$($\emph{left--right measure}$)$ and a multiplicity function, which
together encode the structure of the standard Hilbert space $L^2(M)$ as the natural
$A,A$--bimodule. Both are spectral invariants. Such study of bimodules first appeared in a 
handwritten notes of Connes \cite{Co}. 

\indent For a masa $A\subset M$, one fixes a unital, norm separable and $\sigma$--weakly 
dense $($also dense in the $w.o.t)$ $C^{*}$--subalgebra of 
$A$ which is isomorphic to $C(X)$ for some compact metric $($Hausdorff$)$ space $X$. 
Let $\lambda$ denote the tracial measure on $X$.
For $\zeta_{1},\zeta_{2}\in L^{2}(M)$, let
$\kappa_{\zeta_{1},\zeta_{2}}:C(X)\otimes C(X)\rightarrow \mathbb{C}$ be
the linear functional defined by,
\begin{align}
\nonumber \kappa_{\zeta_{1},\zeta_{2}}(a\otimes b)=\langle
a\zeta_{1} b,\zeta_{2}\rangle,\text{ } a,b\in C(X).
\end{align}
Then $\kappa_{\zeta_{1},\zeta_{2}}$ induces a unique complex Radon
measure $\eta_{\zeta_{1},\zeta_{2}}$ on $X\times X$ given by,
\begin{align}\label{measure_from_kappa}
\kappa_{\zeta_{1},\zeta_{2}}(a\otimes b)=\int_{X\times
X}a(t)b(s)d\eta_{\zeta_{1},\zeta_{2}}(t,s),
\end{align}
and $\norm{\eta_{\zeta_{1},\zeta_{2}}}_{t.v}=\norm{\kappa_{\zeta_{1},\zeta_{2}}}$,
where $\norm{\cdot}_{t.v}$ denotes the total variation norm of
measures.\\
\indent Write $\eta_{\zeta,\zeta}=\eta_{\zeta}$. Note that
$\eta_{\zeta}$ is a positive measure for all $\zeta\in L^{2}(M)$. It
is easy to see that the following polarization type identity holds:
\begin{align}\label{polarize}
4\eta_{\zeta_{1},\zeta_{2}}=\left(\eta_{\zeta_{1}+\zeta_{2}}-\eta_{\zeta_{1}-\zeta_{2}}\right)+i\left(\eta_{\zeta_{1}+i\zeta_{2}}-\eta_{\zeta_{1}-i\zeta_{2}}
\right).
\end{align}
Note that the decomposition of $\eta_{\zeta_{1},\zeta_{2}}$ in equation
\eqref{polarize} need not be its Hahn decomposition in general, but
\begin{align}\label{polarization1}
4\abs{\eta_{\zeta_{1},\zeta_{2}}}&\leq\left(\eta_{\zeta_{1}+\zeta_{2}}+\eta_{\zeta_{1}-\zeta_{2}}\right)+\left(\eta_{\zeta_{1}+i\zeta_{2}}+\eta_{\zeta_{1}-i\zeta_{2}}\right)
=4(\eta_{\zeta_{1}}+\eta_{\zeta_{2}}),
\end{align}
so that
\begin{align}\label{abscont11}
\abs{\eta_{\zeta_{1},\zeta_{2}}}\leq
\eta_{\zeta_{1}}+\eta_{\zeta_{2}}.
\end{align}
For a set $X$, denote by $\Delta(X)$
 the diagonal of $X\times X$.

\begin{Definition}\emph{\cite{DSS,Muk1,Ne-St}}\label{mminv}
The measure--multiplicity invariant of $A\subset M$, denoted by
$m.m(A)$, is the equivalence class of quadruples
$(X,\lambda_{X},[\eta_{\mid\Delta(X)^{c}}],m_{\mid\Delta(X)^{c}})$
under the equivalence relation
$\sim_{m.m}$, where\\
$(i)$ $X$ is a compact Hausdorff space such that $C(X)$ is a unital, norm separable, 
w.o.t dense subalgebra of $A$,\\
$(ii)$ $\lambda_{X}$ is the Borel probability measure obtained by restricting the 
trace $\tau$ on $C(X)$,\\
and\\
$(iii)$ $\eta$ is the measure on $X\times X$ and\\
$(iv)$ $m$ is the multiplicity function,\\
both obtained from the direct integral decomposition of $L^{2}(M)$, so that 
$\A$ is the algebra
of diagonalizable operators with respect to this decomposition; 
the equivalence $\sim_{m.m}$ being,
\begin{align}
\nonumber(X,\lambda_{X},[\eta_{\mid\Delta(X)^{c}}],m_{\mid\Delta(X)^{c}})\sim_{m.m}(Y,\lambda_{Y},[\eta_{\mid\Delta(Y)^{c}}],m_{\mid\Delta(Y)^{c}})
\end{align}
if and only if, there exists a Borel isomorphism $F:X\rightarrow Y$, such
that
\begin{align}
\nonumber & F_{*}\lambda_{X}=\lambda_{Y},\\
\nonumber & (F\times F)_{*}[\eta_{\mid\Delta(X)^{c}}]=[\eta_{\mid\Delta(Y)^{c}}],\\
\nonumber & m_{\mid\Delta(X)^{c}}\circ(F\times
F)^{-1}=m_{\mid\Delta(Y)^{c}}, \text{ }\eta_{\mid\Delta(Y)^{c}} \text{ a.e.}
\end{align}
\end{Definition}

\indent It is easy to see that the Puk\'{a}nszky invariant of
$A\subset M$ is the set of essential values of the multiplicity
function in Definition~\ref{mminv}. The measure class
$[\eta_{\mid\Delta(X)^{c}}]$ is said to be the
\textbf{\emph{left--right measure}} of $A$. Both $m.m(\cdot)$ and $Puk(\cdot)$
are invariants of the masa under automorphisms of the factor $M$. In
this paper, we are mostly interested in the
\emph{left--right measure}.\\
\indent To understand the relation between  mixing properties of masas
and their left--right measures, disintegration of measures will be
used (for a detailed account of disintegration, consult \cite{PC}). Let $T$ be a
measurable map from $(X,\sigma_{X})$ to $(Y,\sigma_{Y})$, where
$\sigma_{X},\sigma_{Y}$ are $\sigma$--algebras of subsets of $X,Y$
respectively. Let $\beta$ be a $\sigma$--finite measure on
$\sigma_{X}$ and $\mu$ a $\sigma$--finite measure on $\sigma_{Y}$.
Here $\beta$ is the measure to be disintegrated and $\mu$ is often
the push forward measure $T_{*}\beta$, although other possibilities
for $\mu$ are allowed.
\begin{Definition}\label{definition_of_disintegration}
We say that $\beta$ has a \emph{disintegration} $\{\beta^{t}\}_{t\in Y}$
with respect to $T$ and $\mu$ or a $(T,\mu)$--disintegration if:\\
$(i)$ $\beta^{t}$ is a $\sigma$--finite measure on $\sigma_{X}$
concentrated on $\{T=t\}$ $($or $T^{-1}\{t\})$, i.e.,
$\beta^{t}(\{T\neq t\})=0$ for
$\mu$--almost all $t$,\\
and, for each nonnegative $\sigma_{X}$--measurable function $f$ on $X$:\\
$(ii)$ $t\mapsto \beta^{t}(f)$ is $\sigma_{Y}$--measurable.\\
$(iii)$
$\beta(f)=\mu^{t}(\beta^{t}(f))\overset{\text{defn}}=\int_{Y}\beta^{t}(f)d\mu(t)$.
\end{Definition}

\indent If $\beta$ in Definition~\ref{definition_of_disintegration} is a complex
measure, then a disintegration of $\beta$ is obtained by first
decomposing it into a linear combination of four positive measures,
using the Hahn decomposition of its real and imaginary parts.     Given a measure $\lambda$ on $X$ and coordinate projections $\pi_i: X \times X \rightarrow X$, $i=1,2,$ we will index the $(\pi_1, \lambda)$-- and $(\pi_2, \lambda)$--disintegrations of a measure on $X \times X$ using superscripts of $t$ and $s$, respectively.  In particular, we will make use of the disintegrations $\{\eta_\zeta^t\}_t$ and $\{\eta_\zeta^s\}_s$ of the measures $\eta_{\zeta_1,\zeta_2}$ defined by equation  \eqref{measure_from_kappa};  these disintegrations are known to exist by \cite[Theorem 3.2]{Muk1} (see also \cite{PC}).  Note that the measure $\eta_\zeta^t$ (respectively, $\eta_\zeta^s$) is concentrated on  $\{t\}\times X$ (respectively, $X\times \{s\}$) for $\lambda$-almost every $t$ (respectively, $\lambda$-almost every $s$).  Denote the restriction of $\eta_\zeta^t$ to $\sett{t} \times X$ by $\tilde{\eta}_\zeta^t$, and the restriction of $\eta_\zeta^s$ to $X \times \sett{s}$ by $\tilde{\eta}_\zeta^s.$

\indent The left--right measure 
of the masa $A$ has the property that if $\theta:X\times X\rightarrow
X\times X$ is the flip map  $\theta(t,s)=(s,t)$, then
$\theta_{*}\eta\ll\eta\ll\theta_{*}\eta$ \cite[Lemma 2.9]{Muk1}. In fact,
it is possible to obtain a choice of $\eta$ for which
$\theta_{*}\eta=\eta$.  Therefore, in most of the following, we will only
state or prove results with respect to the $(\pi_{1},\lambda)$--disintegration; analogous results with respect to the $(\pi_{2},\lambda)$--disintegration are 
also possible. The following lemma, a proof of which can be found in \cite[\S6]{Muk1} and \cite{Muk2},  will be crucial for our purposes.

\begin{Lemma}\label{identify_disintegrated_measure}
Let $\zeta_{1},\zeta_{2}\in L^{2}(M)$ be such that
$\mathbb{E}_{A}(\zeta_{1})=0=\mathbb{E}_{A}(\zeta_{2})$. Let
$\eta_{\zeta_{1},\zeta_{2}}$ denote the Borel measure on
$X\times X$ defined in equation \eqref{measure_from_kappa}.\\
$1^{\circ}$. Then $\eta_{\zeta_{1},\zeta_{2}}$ admits
$(\pi_{i},\lambda)$--disintegrations $X\ni
t\mapsto\eta_{\zeta_{1},\zeta_{2}}^{t}$ and $X\ni
s\mapsto\eta_{\zeta_{1},\zeta_{2}}^{s}$.
Moreover,
\begin{align}
\nonumber \eta_{\zeta_{1},\zeta_{2}}^{t}(X\times
X)=\mathbb{E}_{A}(\zeta_{1}\zeta_{2}^{*})(t), \text{ }\lambda
\text{ a.e.}
\end{align}
$2^{\circ}$. Let $f\in C(X)$. Then the functions $X\ni
t\mapsto \eta_{\zeta_{1},\zeta_{2}}^{t}(1\otimes f), X\ni
s\mapsto \eta_{\zeta_{1},\zeta_{2}}^{s}(f\otimes 1)$ are in
$L^{1}(X,\lambda)$. If $\zeta_{i}\in M$ for $i=1,2$, then $X\ni t\mapsto
\eta_{\zeta_{1},\zeta_{2}}^{t}(1\otimes f), X\ni s\mapsto
\eta_{\zeta_{1},\zeta_{2}}^{s}(f\otimes 1)$ are in
$L^{\infty}(X,\lambda)$.\\
$3^{\circ}$. Let $b,w\in C(X)$. If $\mathbb{E}_{A}(\zeta_{1}
w\zeta_{2}^{*})\in L^{2}(A)$, then
\begin{align}
\nonumber &\norm{\mathbb{E}_{A}(b\zeta_{1} w\zeta_{2}
^{*})}_{2}^{2}=\int_{X}\abs{b(t)}^{2}\abs{{\eta_{\zeta_{1},\zeta_{2}}^{t}(1\otimes
w)}}^{2}d\lambda(t),\\
\nonumber &\norm{\mathbb{E}_{A}(b\zeta_{1} w\zeta_{2}
^{*})}_{1}=\int_{X}\abs{b(t)}\abs{{\eta_{\zeta_{1},\zeta_{2}}^{t}(1\otimes
w)}}d\lambda(t).
\nonumber
\end{align}
\end{Lemma}

\section{Weak Mixing and the Normalizing Algebra}

In this section, we use measure-theoretic techniques to establish several equivalent analytical
conditions that characterize the normalizing algebra of a masa.  These results, along with those in \S5, highlight the relations between mixing, weak mixing of masas and Fourier coefficients 
of mixing and non--atomic measures. The main ingredients in the characterization are 
Theorems $5.5$, $6.6$ of \cite{Muk1} and the following result of Wiener. 

\begin{Theorem}\label{wiener}$($Wiener$)$
Let $\mu$ be a finite Borel measure on $S^{1}$ and let $\widehat{\mu}(n)$, $n\in \mathbb{Z}$, be its Fourier coefficients, i.e., $\widehat{\mu}(n)=\int_{S^{1}}t^{n}d\mu(t)$, $n\in \mathbb{Z}$. Then
\begin{align}
\nonumber \lim_{N\rightarrow\infty}\frac{1}{2N+1}\sum_{n=-N}^{N}\abs{\widehat{\mu}(n)}^{2}=\underset{t\in S^{1}: \mu(\{t\})\neq 0}\sum \text{ }\mu(\{t\})^{2}.
\end{align}
\end{Theorem}
The proof is a direct consequence of the dominated convergence theorem \cite[Lemma 1.1]{Kre}.\\
\indent If $M$ is a fixed II$_1$ factor and $A \subset M$ a masa, let $\lambda$ denote the normalized Haar measure on $S^{1}$ so that
$A\cong L^{\infty}(S^{1},\lambda)$; then $\lambda$ is the tracial
measure. Let $[\eta]$ denote the left--right measure of $A$.  We assume
that $\eta$ is a probability measure on $S^{1}\times S^{1}$, with $\eta(\Delta(S^{1}))=0.$ 
Occasionally in subsequent sections it will be convenient instead to view $A$ as isomorphic to $L^\infty([0,1], \lambda)$, where $\lambda$ is Lebesgue measure on $[0,1]$ (so that $\eta$ would then be a probability on $[0,1] \times [0,1]$).  We will notify the reader in context of any such change.

\begin{Theorem}\label{tfae}
Let $A\subset M$ be a masa. Let $v\in A$ be the Haar unitary generator corresponding to the function 
$S^{1}\ni t\mapsto t\in S^{1}$. Let $x\in M$ be such that $\mathbb{E}_{A}(x)=0$.
Then the following are equivalent.
\begin{align}
\nonumber &(i)\text{  }\lim_{N\rightarrow\infty}\frac{1}{2N+1}\sum_{k=-N}^{N}\norm{\mathbb{E}_{A}(xv^{k}x^{*})}_{2}^{2}=0.\\
\nonumber &(i)^{\prime}\text{  }\lim_{N\rightarrow\infty}\frac{1}{2N+1}\sum_{k=-N}^{N}\norm{\mathbb{E}_{A}(xv^{k}x^{*})}_{2}=0.\\
\nonumber &(ii)\text{  }\lim_{N\rightarrow\infty}\frac{1}{2N+1}\sum_{k=-N}^{N}\norm{\mathbb{E}_{A}(xv^{k}x^{*})}_{1}^{2}=0.\\
\nonumber &(ii)^{\prime}\text{  }\lim_{N\rightarrow\infty}\frac{1}{2N+1}\sum_{k=-N}^{N}\norm{\mathbb{E}_{A}(xv^{k}x^{*})}_{1}=0.\\
\nonumber &(iii) \text{ Given any finite set } \{w_{i}\}_{i=1}^{k}\subset \mathcal{U}(A), \text{ there is a sequence of unitaries } u_{n}\in A, \text{ such that } \\
\nonumber & \indent\lim_{n\rightarrow\infty}\norm{\mathbb{E}_{A}(xw_{i}u_{n}x^{*})}_{2}=0  \text{ for all }1\leq i\leq  k.\\
\nonumber &(iii)^{\prime} \text{ Given any finite set } \{w_{i}\}_{i=1}^{k}\subset \mathcal{U}(A), \text{ there is a sequence of unitaries } u_{n}\in A, \text{ such that } \\
\nonumber & \indent\lim_{n\rightarrow\infty}\norm{\mathbb{E}_{A}(xw_{i}u_{n}x^{*})}_{1}=0 \text{ for all }1\leq i\leq  k.\\
\nonumber &(iv)\text{ }\mathbb{E}_{N(A)^{\prime\prime}}(x)=0.
\end{align}
\end{Theorem}

\begin{proof}
That $(i)\Leftrightarrow (i)^{\prime}$, $(ii)\Leftrightarrow (ii)^{\prime}$ hold, 
follows from the fact that whenever $\{a_{k}\}_{k\in \mathbb{Z}}\subset \mathbb{C}$ is bounded, we have
\begin{align}\label{equilim}
\underset{N\rightarrow \infty}\lim\text{ }\frac{1}{2N+1}\sum_{k=-N}^{N}\abs{a_{k}}=0\text{ }\Leftrightarrow \underset{N\rightarrow \infty}\lim\text{ }\frac{1}{2N+1}\sum_{k=-N}^{N}\abs{a_{k}}^{2}=0. 
\end{align}
Again, $(iii)\Rightarrow (iii)^{\prime}$ and $(i)\Rightarrow (ii)$ are obvious as $\norm{\cdot}_{1}$ is dominated by $\norm{\cdot}_{2}$. Also, $(iii)^{\prime}\Rightarrow (iii)$ follows 
from Lemma \ref{identify_disintegrated_measure} after choosing $u_{n}\in C(S^{1})$ 
$($by making a density argument$)$ and passing to a subsequence.
$(iv)\Rightarrow (i)$, $(iv)\Rightarrow (ii)$ are direct consequences of Theorem $5.5$,  
Theorem $6.6$ and Theorem $6.9$ of \cite{Muk1}, after replacing the base space by $S^{1}$. 
So, we have to prove $(ii)\Rightarrow (iv)$ and 
$(iii)\Leftrightarrow (i)$. \\

$(ii)\Rightarrow (iv)$. Note that $(ii)^{\prime}$ holds. Write 
\begin{align*}
a_{N}= \frac{1}{2N+1}\sum_{k=-N}^{N}\norm{\mathbb{E}_{A}(xv^{k}x^{*})}_{1}, \text{ }
b_{N}=\frac{1}{2N+1}\sum_{k=-N}^{N}\norm{\mathbb{E}_{A}(xv^{k}x^{*})}_{2}^{2}, \forall 
N\in \mathbb{N}. 
\end{align*}
Note that $a_{N}$ and $b_{N}$, $N=1,2,\ldots$ define bounded sequences. There is a constant $C>0$ and a 
set $N_{0}$ with $\lambda(N_{0})=0$ such that $\abs{\eta_{x}^{t}(1\otimes v^{k})}\leq C$ for all 
$t\in N_{0}^{c}$ and all $k\in \mathbb{Z}$. 

We claim that given any subsequence $b_{N_{l}}$, there
is a further subsequence $b_{N_{{l}_{m}}}$ such that $b_{N_{{l}_{m}}}\rightarrow 0$ as $m\rightarrow
\infty$. Then $b_{N}$ would converge to $0$ as $N\rightarrow \infty$. Assume that the claim is true. 

Write $x=y+z$, where $y=\mathbb{E}_{N(A)^{\prime\prime}}(x)$ 
and $z=x-\mathbb{E}_{N(A)^{\prime\prime}}(x)$.
For $a,b\in A$, one has $\langle ayb,z\rangle=0$ and $\langle azb,y\rangle=0$, 
since $AyA\subseteq N(A)^{\prime\prime}$.
It follows that $\eta_{x}=\eta_{y}+\eta_{z}$. Hence, from Lemma 
$3.4$ of \cite{Muk1}, $\eta_{x}^{t}=\eta_{y}^{t}+\eta_{z}^{t}$ for $\lambda$ 
almost all $t$. Write 
\begin{align*}
c_{N}=\frac{1}{2N+1}\sum_{k=-N}^{N}\norm{\mathbb{E}_{A}(zv^{k}z^{*})}_{2}^{2}, \text{ }
 d_{N}=\frac{1}{2N+1}\sum_{k=-N}^{N}\norm{\mathbb{E}_{A}(yv^{k}y^{*})}_{2}^{2}, \forall 
N\in \mathbb{N}.
\end{align*}
Thus, from Lemma \ref{identify_disintegrated_measure},
\begin{align*}
\abs{b_{N}-d_{N}}
\leq&\text{ }\abs{c_{N}+\frac{2}{2N+1}\sum_{k=-N}^{N}\int_{S^{1}}\mathfrak{R}\left(\eta_{z}^{t}(1\otimes v^{k})\eta_{y}^{t}(1\otimes v^{-k})\right)d\lambda(t)}\\
\leq &\text{ }\text{ }c_{N}+ \frac{2}{2N+1}\sum_{k=-N}^{N}\int_{S^{1}}\abs{\eta_{z}^{t}(1\otimes v^{k})\eta_{y}^{t}(1\otimes v^{-k})}d\lambda(t)\\
\leq &\text{ }\text{ }c_{N}+\frac{2}{2N+1}\sum_{k=-N}^{N}\norm{\mathbb{E}_{A}(zv^{k}z^{*})}_{2}\norm{\mathbb{E}_{A}(yv^{k}y^{*})}_{2}\\
\leq &\text{ }\left (1+2\norm{y}^{2}\right)c_{N}.
\end{align*}

\indent Consequently as $\mathbb{E}_{N(A)^{\prime\prime}}(z)=0$, the hypothesis on $x$ and 
Theorem $6.9$ of \cite{Muk1} $($with $[0,1]$ replaced by $S^{1})$ force that
\begin{align}
\nonumber\underset{N}\limsup\text{ }d_{N}=0,
\end{align}
which is a contradiction to Wiener's theorem $($Theorem \ref{wiener}$)$ unless $y=0$, 
from Theorem $5.5$ \cite{Muk1}. Therefore, $y=0$. Thus, we only have to prove the claim.

Fix a subsequence $b_{N_{{l}}}$. Note that $a_{N_{{l}}}\rightarrow 0$ as $l\rightarrow \infty$. So
by Lemma \ref{identify_disintegrated_measure} it follows that
\begin{align*}
\int_{S^{1}}\frac{1}{2N_{{l}}+1}\sum_{k=-N_{{l}}}^{N_{{l}}} \abs{\eta_{x}^{t}(1\otimes v^{k})}d\lambda(t) \rightarrow 0, \text{ as }l\rightarrow \infty.
\end{align*}
Dropping to a subsequence $b_{N_{{l_{m}}}}$, replacing the null set $N_{0}$ by a $($probably$)$ larger null set if necessary and renaming it to be 
$N_{0}$ again, it follows that 
\begin{align*}
 \frac{1}{2N_{{l_{m}}}+1}\sum_{k=-N_{{l_{m}}}}^{N_{{l_{m}}}} \abs{\eta_{x}^{t}(1\otimes v^{k})}\rightarrow 0, \text{ as }m\rightarrow\infty\text{ for all }t\in N_{0}^{c}.
\end{align*}
Now for $t\in N_{0}^{c}$, 
\begin{align*}
 \frac{1}{2N_{{l_{m}}}+1}\sum_{k=-N_{{l_{m}}}}^{N_{{l_{m}}}} \abs{\eta_{x}^{t}(1\otimes v^{k})}^{2}\leq C \frac{1}{2N_{{l_{m}}}+1}\sum_{k=-N_{{l_{m}}}}^{N_{{l_{m}}}} \abs{\eta_{x}^{t}(1\otimes v^{k})}, \text{ for all }m.
\end{align*}
Direct application of dominated convergence theorem and Lemma \ref{identify_disintegrated_measure} 
shows that $b_{N_{l_{m}}}\rightarrow 0$ as $m\rightarrow \infty$.\\

\noindent $(iii)\Rightarrow (i)$. By making a density argument, we can assume that 
there is a sequence of unitaries $u_{n}\in C(S^{1})\subset A$, such that 
$\norm{\mathbb{E}_{A}(xw_{i}u_{n}x^{*})}_{2}\rightarrow 0$ as $n\rightarrow \infty$ 
for all $1\leq i\leq k$. We will only show that $\tilde{\eta}_{x}^{t}$ is non--atomic 
for $\lambda$ almost all $t$. Then, in view of Theorem $5.5$ and Theorem $6.6$ of 
\cite{Muk1}, the arguments are complete.\\
\indent For each $l\in \mathbb{N}$, choose a unitary $u_{l}\in C(S^{1})$, such that
\begin{align}
\nonumber \norm{\mathbb{E}_{A}(xv^{j}u_{l}x^{*})}_{2}< \frac{1}{l+1}, \text{ }j=0,\pm 1,\cdots,\pm l.
\end{align}
Lemma \ref{identify_disintegrated_measure} yields
\begin{align}
\nonumber
\norm{\mathbb{E}_{A}(xv^{j}u_{l}x^{*})}_{2}^{2}=\int_{S^{1}}\abs{\eta_{x}^{t}(1\otimes v^{j}u_{l})}^{2}d\lambda(t)<
\frac{1}{(l+1)^{2}},\text{ }j=0,\pm 1,\cdots,\pm l, \text{ for all }l.
\end{align}
\noindent Therefore,
\begin{align}
\nonumber
\underset{l}\lim\int_{S^{1}}\abs{\eta_{x}^{t}(1\otimes v^{j}u_{l})}^{2}d\lambda(t)=0, \text{ for }j=0,\pm 1,\pm 2, \cdots,\pm N, \text{ for all }N\in \mathbb{N}.
\end{align}
Using Cantor's diagonal argument, we may extract a subsequence $l_{p}<l_{p+1}$ for all $p$, and a set
$F\subset S^{1}$ with $\lambda(F)=0$, such that for all $t\in F^{c}$,
\begin{align}
\nonumber\underset{p}\lim\int_{S^{1}}s^{j}u_{l_{p}}(s)
d\tilde{\eta}_{x}^{t}(s)=0
\end{align}
for $j=0,\pm 1,\pm 2, \cdots$, and $\tilde{\eta}_{x}^{t}$ is a
finite measure $($see Lemma \ref{identify_disintegrated_measure}$)$. Consequently, by the
Stone--Weierstrass theorem, we have for all $f\in C(S^{1})$,
\begin{align}\label{cont}
\underset{p}\lim\int_{S^{1}}f(s)u_{l_{p}}(s)d\tilde{\eta}_{x}^{t}(s)=0, \text{ for }t\in F^{c}.
\end{align}
A further density argument establishes that
equation \eqref{cont} holds if $f$ is the indicator function of a compact set. It follows that
$\tilde{\eta}_{x}^{t}$ cannot have any atoms for $t\in F^{c}$.\\
\noindent $(i)\Rightarrow (iii)$. As $(i)\Leftrightarrow (iv)$ so $\mathbb{E}_{N(A)^{\prime\prime}}(x)=0$, and hence
$\mathbb{E}_{N(A)^{\prime\prime}}(xw_{i})=0$ for all $1\leq i\leq k$. Thus 
$\tilde{\eta}_{xw_{i}}^{t}$ and $\tilde{\eta}_{x}^{t}$ are non--atomic for all 
$1\leq i\leq k$ and for $\lambda$ almost all $t$. Use equation \eqref{polarize} to conclude that 
$\frac{1}{2N+1}\sum_{k=-N}^{N}\abs{\eta_{xw_{i},x}^{t}(1\otimes v^{k})}^{2}$ goes to zero as 
$N\rightarrow \infty$ almost everywhere $\lambda$ for all $1\leq i\leq k$. Now use Lemma 
\ref{identify_disintegrated_measure} to conclude that 
$\frac{1}{2N+1}\sum_{k=-N}^{N}\norm{\mathbb{E}_{A}(xw_{i}v^{k}x^{*})}_{2}^{2}\rightarrow 0$ for all $i$.
Thus, there is a set $S\subset \mathbb{Z}$ of density one such that $\norm{\mathbb{E}_{A}(xw_{i}v^{k}x^{*})}_{2}$ goes to zero as $\abs{k}\rightarrow \infty$ along $S$ \cite{Muk1}. This completes the proof.
\end{proof}

\begin{Remark}
\emph{Note that} \emph{$(iii)\Rightarrow (iv)$ in Theorem \ref{tfae}} \emph{is to be compared with Lemma 2.5 in\cite{Po1}}.
\end{Remark}

\begin{Corollary}$($Independence of coordinates$)$\label{incoord}
Let $A\subset M$ be a masa. Let $v\in A$ be the Haar unitary generator corresponding to the function 
$S^{1}\ni t\mapsto t\in S^{1}$. Let $x\in M$ be such that $\mathbb{E}_{A}(x)=0$. Then the following are equivalent.
\begin{align}
\nonumber &(i)\text{  }\lim_{N\rightarrow\infty}\frac{1}{2N+1}\sum_{k=-N}^{N}\norm{\mathbb{E}_{A}(xv^{k}x^{*})}_{2}^{2}=0.\\
\nonumber &(ii)\text{  }\lim_{N\rightarrow\infty}\frac{1}{2N+1}\sum_{k=-N}^{N}\norm{\mathbb{E}_{A}(xw^{k}x^{*})}_{2}^{2}=0, \text{ for any Haar unitary generator } w \text{ of }A.
\end{align}
\end{Corollary}

\begin{Remark}
\emph{Note that $(i)$ in Theorem \ref{tfae} is false for any Haar unitary of $A$. 
There can be diffuse subalgebras inside $A$ with large normalizers. For example,
consider the masa $A(1)$ in \cite{SW}. The averaging 
conditions in Theorem \ref{tfae} are the analogues of weakly mixing actions. } 
\end{Remark}

\begin{Remark}
\emph{Rigidity and non--recurrence are two properties of dynamical 
systems that are in some sense opposite to each other \cite{Ber}. When translated to
the language of operator algebras, rigidity characterizes masas having non trivial central 
sequences, while non--recurrence is close to mixing. The sequences along which these properties
occur for weakly mixing transformations have rich structure \cite{Ber}. Thus, in view of
the results in \cite{Ber}, it is important to know more of the asymptotic properties of 
$\mathbb{E}_{A}(xv^{k}x^{*})$, $k\in \mathbb{Z}$. The operators $\mathbb{E}_{A}(xv^{k}x^{*})$, $k\in \mathbb{Z}$,
are directly related to Fourier coefficients of certain measures that characterize (weak) mixing and
rigidity $($see proof of Theorem \ref{action_sm}$)$.} 
\end{Remark}

\begin{Theorem}\label{CorChanged}
Let $A\subset M$ be a masa and let $v\in A$ be the Haar unitary generator of $A$ 
corresponding to the function $S^{1}\ni t\mapsto t\in S^{1}$. There is a subsequence 
$k_{l}$ $(k_{l}<k_{l+1})$ such that $\norm{\mathbb{E}_{A}(xv^{k_{l}}x^{*})}_{2}\rightarrow 0$ 
as $l\rightarrow \infty$ for all $x\in M$ such that $\mathbb{E}_{N(A)^{\prime\prime}}(x)=0$.\end{Theorem}

\begin{proof}
The proof follows easily from Theorem \ref{tfae} by separability and a diagonalization
argument. We omit the details.
\end{proof}

\begin{Remark}
\emph{ When compared with the asymptotic homomorphism property $($AHP$)$ introduced in \cite{S-S3}, and Proposition $1.1$ of \cite{J-S}, we suspect that Theorem \ref{CorChanged} may be the best result along these lines that can be expected in general.  Moreover, by making an argument appealing to Corollary \ref{incoord} and Theorem \ref{tfae}, 
it can be shown that the statement in Theorem \ref{CorChanged} $($except possibly the subsequence$)$ is independent of the choice of Haar unitary generator. In other words, any singular masa `almost has the AHP' with respect to any choice of its Haar unitary generator.}
\end{Remark}

\section{Masas of Product Class}

A masa $A$ in a II$_1$ factor $M$ is said to be of product class if its left-right measure is the class of product measure.  This condition was shown in \cite{Muk2} to be equivalent to the condition that the space $L^{2}(M) \ominus L^2(A)$ decomposes as a direct sum of coarse $A$--$A$ bimodules.   Masas with this property were studied in detail in \cite{Muk2}, though in essence they have been known about for  some time.
In this section, we study masas of product class in the context of mixing properties of subalgebras. One of our main results in this section builds on those obtained by Jolissaint and Stalder in \cite{J-S}, in particular, \cite[Theorem 3.5]{J-S}.
Furthermore, in restricting ourselves to a smaller class of dynamical systems -- namely, the ones that arise from semidirect products of groups -- the absence of $\mathbb{Z}$--systems with strictly 
absolutely continuous spectrum will be a consequence of the analysis undertaken in \S4 and \S5.  We begin by recalling the following property of masas from \cite{Muk2}, which is closely related to masas of product class.

\begin{Definition}  \label{SU}   We say that a masa $A\subset M$ has the property (SU) if it satisfies the following conditions:  There exists a set $S\subset
M$ such that $\mathbb{E}_{A}(x)=0$ for all $x\in S$, and
the linear span of $S$ \text{ is dense in }$L^{2}(M)\ominus L^{2}(A)$; there is an  
orthonormal basis
$\{v_{n}\}_{n=1}^{\infty}\subset A$ of $L^{2}(A)$ such that
\begin{align}\label{unimix}
\sum_{n=1}^{\infty}\norm{\mathbb{E}_{A}(x
v_{n}x^{*})}_{2}^{2}<\infty \text{ for all }x\in S;
\end{align}
and there is a nonzero vector $\zeta\in L^{2}(M)\ominus
L^{2}(A)$ such that $\mathbb{E}_{A}(\zeta u^{n}\zeta^{*})=0$ for all
$n\neq 0$, where $u$ is a Haar unitary generator of $A$.
\end{Definition}

\indent There are many examples of masas in both the hyperfinite and free group factors known to satisfy $(\text{SU})$. In \cite{Muk2}, it was shown that any masa satisfying (SU) is also a masa of product class.  In the same work, it was shown that masas of product class satisfy a similar (but slightly weaker) set of conditions to those in (SU) \cite[Theorem $2.5$]{Muk2}.    In the analysis that follows,
we assume $A=L^{\infty}([0,1],\lambda)$. Note that the results below use Lemma \ref{Lemma:appendix} from the appendix.

\begin{Theorem}\label{summability_imply_mixing}
Let $A\subset M$ be a masa of product class. Then $A$ is a  mixing masa.
\end{Theorem}

\begin{proof}
In Theorem 2.5 \cite{Muk2}, it was shown that the hypothesis implies the
following. There is a set $S\subset L^{2}(M)\ominus L^{2}(A)$ such
that $\text{ span }S$ is dense in $L^{2}(M)\ominus L^{2}(A)$,
\begin{align}
\nonumber \sum_{n\in \mathbb{Z}}\norm{\mathbb{E}_{A}(\zeta
v^{n}\zeta^{*})}_{2}^{2}<\infty \text{ for all }\zeta\in S,
\end{align}
and there is a nonzero $\xi_{0}\in L^{2}(M)\ominus L^{2}(A)$
such that $\mathbb{E}_{A}(\xi_{0} v^{n}\xi_{0}^{*})=0$ for all $n\neq 0$,
where $v$ is the Haar unitary generator of $A$ corresponding to the function 
$[0,1]\ni t\mapsto e^{2\pi it}$. Furthermore, the proof of the same theorem 
shows that $S$ can be chosen so that $\frac{d\eta_{\zeta}}{d(\lambda\otimes\lambda)}$ 
is essentially bounded for $\zeta\in S$. \\
\indent Use Lemma \ref{Lemma:appendix} and Remark \ref{Rem:append} to conclude that 
there is a dense subset $S^{\prime}\subset L^{2}(M)\ominus L^{2}(A)$ such that
\begin{align}
\nonumber \sum_{n\in \mathbb{Z}}\norm{\mathbb{E}_{A}(\xi_{1}
v^{n}\xi_{2}^{*})}_{2}^{2}<\infty \text{ for all }\xi_{1},\xi_{2}\in S^{\prime},
\end{align}
and $\frac{d\eta_{\xi}}{d(\lambda\otimes \lambda)}$ is essentially bounded 
for all $\xi\in S^{\prime}$. In the above statements, it is implicit that the vectors 
$\zeta$, $\xi_{i}$, $i=1,2$, are such that $\mathbb{E}_{A}(\zeta
v^{n}\zeta^{*}),\mathbb{E}_{A}(\xi_{1}v^{n}\xi_{2}^{*})\in L^{2}(A)$; thus there is 
no confusion in considering their $L^{2}$--norms. \\
\indent Making arguments as in \cite[Section 11.4]{S-S2}, it is easy to see
that if $\{a_{n}\}\subset A$ is a bounded weakly null sequence of operators,
then $\mathbb{E}_{A}(\xi_{1} a_{n}\xi_{2}^{*})\rightarrow 0$ in
$\norm{\cdot}_{2}$ for all $\xi_{1},\xi_{2}\in S^{\prime}$.\\
\indent Fix $x\in M$ such that $\mathbb{E}_{A}(x)=0$. Also fix $\xi\in S^{\prime}$ 
and a weakly null sequence of operators $\{a_{n}\}\subset A$ in the unit ball. 
Choose a sequence of vectors $\zeta_{k}\in S^{\prime}$ such that
$\zeta_{k}\rightarrow x$ in $\norm{\cdot}_{2}$. For $\xi\in S^{\prime}$, from  
Lemma \ref{identify_disintegrated_measure} we have 
$\mathbb{E}_{A}(\xi\xi^{*})(t)=\eta_{\xi}^{t}([0,1]\times [0,1])=\int_{0}^{1}f_{\xi}(t,s)d\lambda(s)$ for
$\lambda$ almost all $t$, where $f_{\xi}=\frac{d\eta_{\xi}}{d(\lambda\otimes \lambda)}$. 
Since $f_{\xi}$ is essentially bounded, it follows from Lemma $3.6$ \cite{Muk1} that $\mathbb{E}_{A}(\xi\xi^{*})\in L^{\infty}([0,1],\lambda)$.
Thus, for $n\in \mathbb{N}$,
\begin{align}
\nonumber \underset{a\in C[0,1]:\norm{a}_{2}\leq1}\sup
\abs{\int_{0}^{1}a(t)\mathbb{E}_{A}((\zeta_{k}-x)a_{n}\xi^{*})(t)d\lambda(t)}&=\underset{a\in C[0,1]:\norm{a}_{2}\leq1}{\sup}\abs{\tau((\zeta_{k}-x)a_{n}\xi^{*}a)}\\
\nonumber &=\underset{a\in C[0,1]:\norm{a}_{2}\leq1}\sup\abs{\langle (\zeta_{k}-x)a_{n},a^{*}\xi\rangle}\\
\nonumber &\leq\norm{\zeta_{k}-x}_{2}\underset{a\in C[0,1]:\norm{a}_{2}\leq1}\sup \langle a^{*}\xi,a^{*}\xi\rangle^{\frac{1}{2}}\\
\nonumber &=\norm{\zeta_{k}-x}_{2}\underset{a\in C[0,1]:\norm{a}_{2}\leq1}\sup \tau(a^{*}\mathbb{E}_{A}(\xi\xi^{*})a)^{\frac{1}{2}}\\
\nonumber &\leq \norm{\mathbb{E}_{A}(\xi\xi^{*})}^{\frac{1}{2}}\norm{\zeta_{k}-x}_{2}.
\end{align}
This shows that $\mathbb{E}_{A}((\zeta_{k}-x)a_{n}\xi^{*})\in L^{2}(A)$ for all $k,n$, and
a triangle inequality argument shows that $\mathbb{E}_{A}(xa_{n}\xi^{*})\rightarrow 0$ 
in $\norm{\cdot}_{2}$. Make a further density argument to finish the proof. 
We omit the details.
\end{proof}

\begin{Remark}
\emph{A similar argument along with Theorem 3.1 \cite{SSLap} gives a proof of the 
fact that the radial $($Laplacian$)$ masa in $L(\mathbb{F}_{k})$, $2\leq k<\infty$ is
mixing. The same can be deduced as a corollary of Theorem \ref{summability_imply_mixing}, 
as the left--right measure of the radial masa is the class of product measure \cite{DyMu}.}
\end{Remark}

The next result shows that masas of product class can in fact be compared directly with mixing masas, and possess far stronger convergence properties.  

\begin{Theorem}\label{aeconv}
Let $A\subset M$ be a masa. Suppose the left--right measure of $A$ is
the class of product measure. Let $x,y\in M$ be such that
$\mathbb{E}_{A}(x)=0=\mathbb{E}_{A}(y)$. If $(u_{n})$ is a
bounded sequence in $A$ converging to zero in the weak operator topology, then
$\mathbb{E}_{A}(xu_{n}y^{*})$ converges to zero $\lambda$ almost
everywhere.
\end{Theorem}

Before we prove Theorem~\ref{aeconv}, we need to make an observation.
Let $x\in M$ be such that $\mathbb{E}_{A}(x)=0$. In the results of the third author in \cite{Muk1,Muk2}
that involved disintegration of measures, it was necessary to work with functions of the form
$[0,1]\ni t\mapsto \eta_{x}^{t}(1\otimes a)$, where $a\in
C[0,1]\subset A$ $($or $a\in C(S^{1})\subset A$ as the case may be$)$. The reason for the choice of $a\in C[0,1]$ $($or $a\in C(S^{1}))$ in that work was to ensure that the function $[0,1]\ni
t\mapsto \eta_{x}^{t}(1\otimes a)$ was finite almost everywhere and measurable. However, if $[\eta]=[\lambda\otimes \lambda]$, then we can allow $a$ to be in
$L^{\infty}([0,1],\lambda)$ $($or $L^{\infty}(S^{1},\lambda))$. In this
case, the aforementioned finiteness and measurability are not issues.

\begin{proof}[Proof of Theorem~\ref{aeconv}]
First, fix $x\in M$ with $\mathbb{E}_{A}(x)=0$. Note that
$\eta_{x}\ll \lambda\otimes \lambda$ \cite[Lemma $5.7$]{DSS}. Let
$g=\frac{d\eta_{x}}{d(\lambda\otimes \lambda)}$. Then $g\in
L^{1}(\lambda\otimes \lambda)$. From Lemma 3.6 \cite{Muk1},
$\tilde{\eta}_{x}^{t}\ll\lambda$ and
$\frac{d\tilde{\eta}_{x}^{t}}{d\lambda}=g_{t}$ for $\lambda$ almost
all $t$, where $g_{t}= g(t,\cdot)$.\\
\indent It is easy to verify that, $[0,1]\ni t\mapsto
\eta_{x}^{t}(1\otimes u_{n})$ is in $L^{\infty}([0,1],\lambda)$ for
all $n$ $($use Lemma \ref{identify_disintegrated_measure}$)$. For $a\in C[0,1]$, the equation
\begin{align}\label{identfy_Fcoeff}
\langle \mathbb{E}_{A}(xu_{n}x^{*}),a\rangle=\tau(a^{*}\mathbb{E}_{A}(xu_{n}x^{*}))=\tau(a^{*}xu_{n}x^{*})=\int_{0}^{1}\overline{a(t)}\eta_{x}^{t}(1\otimes
u_{n})d\lambda(t)
\end{align}
implies that, $\mathbb{E}_{A}(xu_{n}x^{*})(t)=\eta_{x}^{t}(1\otimes
u_{n})$ for $\lambda$ almost all $t$. Thus, for $\lambda$ almost all
$t$ we have,
\begin{align}
\nonumber \mathbb{E}_{A}(xu_{n}x^{*})(t)&=\eta_{x}^{t}(1\otimes
u_{n})=\int_{0}^{1}u_{n}(s)g_{t}(s)d\lambda(s)\rightarrow 0 \text{
as }n\rightarrow \infty.
\end{align}
The last statement holds because $\{u_{n}\}$ is bounded, converges to zero in $w.o.t$ 
and $g_{t}\in L^{1}(\lambda)$ for almost all $t$.\\
\indent Finally use the identity,
\begin{align}
\nonumber 4\text{ }\mathbb{E}_{A}(xu_{n}y^{*})&=\mathbb{E}_{A}((x+y)u_{n}(x+y)^{*})-\mathbb{E}_{A}((x-y)u_{n}(x-y)^{*})\\
\nonumber &+ i\text{ }\mathbb{E}_{A}((x+iy)u_{n}(x+iy)^{*})-i\text{
}\mathbb{E}_{A}((x-iy)u_{n}(x-iy)^{*}), \text{ for all }n,
\end{align}
to complete the proof.
\end{proof}

\begin{Remark}
\emph{Observe that in Theorem~\ref{aeconv}, even if we assume only that the left-right 
measure is absolutely continuous with respect to the product measure, then 
the conclusions remain valid.}
\end{Remark}

The following lemma is likely to be well-known to experts, but we lack a reference, so we present 
it here for convenience.

\begin{Lemma}\label{haar_compare}
Let $H$ be a countable discrete torsion-free abelian group. Let $G$ be a closed subgroup of $\widehat{H}$
such that the normalized Haar measure $\lambda_{G}$ of $G$ is absolutely continuous with respect to the 
normalized Haar measure $\lambda_{\widehat{H}}$ of $\widehat{H}$ 
$($regarding $\lambda_{G}$ to be a measure on $\widehat{H}$ by extending it by zero on the complement 
of $G)$. Then $G=\widehat{H}$. 
\end{Lemma}

\begin{proof}
Note that both $\widehat{H}$ and $G$ are compact abelian groups, thus  $\lambda_{G}$ and 
$\lambda_{\widehat{H}}$ exist. Consequently, $\lambda_{\widehat{H}}(G)>0$. 
Since $H$ is torsion-free, so $\widehat{H}$ is connected 
$($see Theorem $24. 25$ \cite{HR}$)$.  
Since $\lambda_{\widehat{H}}$ is normalized, the translation invariance of Haar measure
forces that $\widehat{H}/G$ is a finite abelian group. However, the quotient map $q:\widehat{H}
\rightarrow \widehat{H}/G$ is continuous, so the image of $q$ is connected. Thus, $\widehat{H}/G$
is trivial.
\end{proof}

\begin{Theorem}\label{inclusion_of_subgroups}
Let $\Gamma$ be a countable discrete non abelian group and let $\Gamma_{0}$ be an infinite
abelian subgroup of $\Gamma$. Suppose $L(\Gamma_{0})\subset
L(\Gamma)$ is a  mixing masa. Then,\\
$(i)$ the left--right measure of $L(\Gamma_{0})$ is the class of a measure which is 
absolutely continuous with respect to the product measure;\\
$(ii)$ if $\Gamma_{0}$ is torsion-free, then the left--right 
measure of $L(\Gamma_{0})$ is the class of product measure. 
Moreover, $\Gamma$ is i.c.c. $($infinite conjugacy class$)$ and 
$Puk_{L(\Gamma)}(L(\Gamma_{0}))=\{m\}$, $m\in \mathbb{N}\cup \{\infty\}$ 
and $\Gamma_{0}$ is malnormal in $\Gamma$.
\end{Theorem}
\begin{proof}
We first prove the statements regarding the left--right measures in both cases and then prove the remaining 
statements of $(ii)$. In \cite[Theorem 3.5]{J-S}, it was shown that the hypothesis is equivalent
to the following condition: $($ST$)$ For every finite subset $F\subset \Gamma \setminus
\Gamma_{0}$, there exists a finite subset $E$ of $\Gamma_{0}$ such
that $gg_{0}h\not\in \Gamma_{0}$ for all $g_{0}\in\Gamma_{0}\setminus E$ and all $g,h\in F$. \\
\indent Let $\widehat{\Gamma_{0}}$ denote the Pontryagin dual of $\Gamma_{0}$ and let $\lambda_{\widehat{\Gamma_{0}}}$ denote the normalized Haar measure on $\widehat{\Gamma_{0}}$. The
left--right measure of $L(\Gamma_{0})$ is naturally supported
on $\widehat{\Gamma_{0}}\times \widehat{\Gamma_{0}}$. Let $u_{g}\in L(\Gamma)$ be the unitary
operator corresponding to the group element $g\in \Gamma$. Fix $g\in
\Gamma\setminus\Gamma_{0}$. Then, taking $F=\{g,g^{-1}\}$, there is a
finite subset $E$ of $\Gamma_{0}$ such that
$\mathbb{E}_{L(\Gamma_{0})}(u_{g}u_{h}u_{g}^{*})=0$ for all
$h\in\Gamma_{0}\setminus E$. Therefore, by arguments similar to those in the proof of Theorem \ref{aeconv} (using, in particular, the appropriate analogue of equation \eqref{identfy_Fcoeff}), we get 
\begin{align}
\nonumber\eta_{u_{g}}^{t}(1\otimes
\check{h})=0 \text{ for } \lambda_{\widehat{\Gamma_{0}}} \text{ almost all }t\in 
\widehat{\Gamma_{0}} \text{ and }h\in
\Gamma_{0}\setminus E,
\end{align} 
where $\check{h}$ is the canonical image of
$h$ in $C(\widehat{\Gamma_{0}})$. Recall that $\left\{u_{h}:h\in \Gamma_{0}\right\}$ is an orthonormal basis of 
$\ell^{2}(\Gamma_{0})$. Thus, one has
\begin{align}
\nonumber \sum_{h\in \Gamma_{0}}\norm{\mathbb{E}_{L(\Gamma_{0})}(u_{g}u_{h}u_{g}^{*})}^{2}_{2}<\infty.
\end{align}
From the proof of Proposition 2.4 of \cite{Muk2} and the remark following it, we get 
that $\tilde{\eta}_{u_{g}}^{t}\ll\lambda_{\widehat{\Gamma_{0}}}$ and the Radon--Nikodym 
derivative $f_{t}$ of $\tilde{\eta}_{u_{g}}^{t}$ with respect to 
$\lambda_{\widehat{\Gamma_{0}}}$ is in $L^{2}(\lambda_{\widehat{\Gamma_{0}}})$ 
for $\lambda_{\widehat{\Gamma_{0}}}$ almost all $t$. However, since  
$\left\{\check{h}:h\in \Gamma_{0}\right\}$ is an orthonormal basis of 
$L^{2}(\lambda_{\widehat{\Gamma_{0}}})$, one has 
$f_{t}=\sum_{h\in \Gamma_{0}}\langle f_{t},\check{h}\rangle \check{h}$ for 
$\lambda_{\widehat{\Gamma_{0}}}$ almost all $t$ and the series converge in 
$L^{2}(\lambda_{\widehat{\Gamma_{0}}})$. Consequently, only finitely many terms of 
the Fourier series survive and hence $f_{t}$ is continuous for 
$\lambda_{\widehat{\Gamma_{0}}}$ 
almost all $t$. \\
\indent If $\Gamma_{0}$ is torsion-free, then so is $\Gamma_{0}\times \Gamma_{0}$. 
By Lemma $5.6$ of \cite{DSS}, $\eta_{u_{g}}$ is the normalized Haar measure
of the subgroup $K_{g}^{\circ}$, where $K_{g}=\{(h_{1},h_{2})\in \Gamma_{0}\times \Gamma_{0}: h_{1}gh_{2}=g\}$
and $K_{g}^{\circ}=\{\gamma\in \widehat{\Gamma_{0}}\times \widehat{\Gamma_{0}}: \gamma(K_{g})=1\}$. 
By the first part of the argument 
$\eta_{u_{g}}\ll \lambda_{\widehat{\Gamma_{0}}}\otimes \lambda_{\widehat{\Gamma_{0}}}$. Thus,  
by Lemma \ref{haar_compare} it follows that $K_{g}^{\circ}=\widehat{\Gamma_{0}}\times \widehat{\Gamma_{0}}$,
i.e., $K_{g}$ is trivial. \\
\indent No matter what $\Gamma_{0}$ be, if $g_{i}\in \Gamma\setminus\Gamma_{0}$ and 
$c_{i}\in \mathbb{C}$ for $1\leq i\leq n$, then 
$\eta_{\sum_{i=1}^{n}c_{i}u_{g_{i}}}=\sum_{i=1}^{n}\abs{c_{i}}^{2}\eta_{u_{g_{i}}}+\sum_{i\neq j=1}^{n}c_{i}\bar{c}_{j}\eta_{u_{g_{i}},u_{g_{j}}}\ll \lambda_{\widehat{\Gamma_{0}}}\otimes \lambda_{\widehat{\Gamma_{0}}}$ from equation \eqref{abscont11}. Now, the linear span of $\left\{u_{g}:g\in \Gamma\setminus
\Gamma_{0}\right\}$ is dense in $\ell^{2}(\Gamma_{0})^{\perp}$ in
$\norm{\cdot}_{2}$. Use Lemma 3.9, 3.10 of \cite{Muk1} to conclude that
$\eta_{\zeta}\ll \lambda_{\widehat{\Gamma_{0}}}\otimes \lambda_{\widehat{\Gamma_{0}}}$ 
for all $\zeta\in \ell^{2}(\Gamma_{0})^{\perp}$. Finally, use Lemma 5.7 \cite{DSS} to 
conclude about the left--right measure. \\
\indent The statement regarding the Puk\'{a}nzsky invariant in case $(ii)$ follows directly 
from Lemma 5.7 \cite{DSS} and the preceding arguments, as $K_{g}=K_{g^{\prime}}$ 
for all $g,g^{\prime}\in \Gamma\setminus\Gamma_{0}$. 
$($The same can be directly deduced from Theorem 4.1 \cite{S-S4} as well by considering the double coset 
structure of $\Gamma_{0}$ in $\Gamma)$. Malnormality of $\Gamma_{0}$ is true for the same 
reason $(K_{g}$ is trivial for $g\not\in \Gamma_{0})$.\\
\indent We now show that $L(\Gamma)$ is a factor $($in case $(ii)$ of the statement$)$, 
which will force $\Gamma$ to be i.c.c. If $p\neq 0$ is a central projection of $L(\Gamma)$, then 
$L(\Gamma)=L(\Gamma)p\oplus L(\Gamma)(1-p)$. Note that $p\in L(\Gamma_{0})$. 
The left--right measure of $L(\Gamma_{0})$ is 
$[\lambda_{\widehat{\Gamma_{0}}}\otimes \lambda_{\widehat{\Gamma_{0}}}]$ 
and $[\eta_{u_{g}}]=[\lambda_{\widehat{\Gamma_{0}}}\otimes \lambda_{\widehat{\Gamma_{0}}}]$ for $g\in\Gamma\setminus\Gamma_{0}$. But for any $a,b\in L(\Gamma_{0})$, we have
\begin{align}
\nonumber \langle au_{g}b,u_{g}\rangle&= \langle (ap\oplus a(1-p))(u_{g}p\oplus u_{g}(1-p))(bp\oplus b(1-p)),(u_{g}p\oplus u_{g}(1-p))\rangle\\
\nonumber &=\langle apu_{g}pbp,u_{g}p\rangle+\langle a(1-p)u_{g}(1-p)b(1-p),u_{g}(1-p)\rangle.
\end{align}
This shows that $\eta_{u_{g}}$ is supported on the union of two measurable 
rectangles and hence the left--right measure is concentrated on the same set.  
Consequently, the left--right measure cannot be equal to the product class unless $p=1$. 
Thus, $L(\Gamma)$ is a factor.
\end{proof}

\begin{Remark}
\emph{As was pointed out to us by the referee of an earlier version of this paper, $(ii)$ of Theorem \ref{inclusion_of_subgroups} does not hold if 
$\Gamma_{0}$ has torsion. For example, consider the inclusion $\mathbb{Z}\times \mathbb{Z}/n\mathbb{Z}\subset \mathbb{F}_{2}\times \mathbb{Z}/n\mathbb{Z}$, where $\mathbb{Z}$ is a free factor of $\mathbb{F}_{2}$. Then the 
hypothesis of Theorem \ref{inclusion_of_subgroups} is satisfied but the left--right measure of the inclusion is 
absolutely continuous but not Lebesgue, as $\mathbb{F}_{2}\times \mathbb{Z}/n\mathbb{Z}$ is not i.c.c. Also, in the
torsion-free case malnormality of $\Gamma_{0}$ is equivalent to strong mixing, for malnormality forces the masa to be of product class. Malnormal subgroups were used by Popa to gain 
control over normalizers and relative commutants \cite{Po1}.  Recently, Robertson 
and Steger \cite{RobSt} have proved that, if
$G$ is a connected semisimple real algebraic group such that
$G(\mathbb{R})$ has no compact factors, then any torsion-free
uniform lattice subgroup $\Gamma$ of $G(\mathbb{R})$ contains a
malnormal abelian subgroup $\Gamma_{0}$ such that
$Puk_{L(\Gamma)}(L(\Gamma_{0}))=\{\infty\}$. Thus, the group von
Neumann factor of any such torsion-free uniform lattice subgroup
contains a singular masa of product class and infinite multiplicity.
In general, it is of interest to know whether every $\rm{II}_{1}$
factor has a singular masa of product class and infinite
multiplicity. It is also to be noted that, there are no examples so far of masas in $\rm{II}_{1}$ 
factors for which the left--right measure is absolutely continuous with respect to 
the product class, but not equivalent to the product class.}
\end{Remark}

\indent The significance of the next corollary will become clear in the next section, when we relate the left--right 
measure to the maximal spectral type of an action.

\begin{Corollary}\label{notexist}
There does not exist any countable discrete group of the form 
$\Gamma\rtimes_{\alpha} \Gamma_{0}$ with $\abs{\Gamma}=\infty$, and $\Gamma_{0}$ being torsion-free and abelian, such that $L(\Gamma_{0})\subset L(\Gamma)\rtimes_{\alpha}\Gamma_{0}$ 
is a mixing masa for which the left--right measure is absolutely continuous with 
respect to the product measure but not equivalent to the product measure.
\end{Corollary}

\indent The next result is in spirit similar to the results in \cite{Kre} regarding 
wandering vectors. The results in \cite{Kre} are statements about modules 
over abelian von Neumann algebras, while the next result deals with bimodules.  
Theorem \ref{density_wand_vect} precisely generalizes the `malnormality condition'
$($in the context of group inclusions$)$ for masas of product class.

\begin{Theorem}\label{density_wand_vect}
Let $A\subset M$ be a masa of product class. Let $v\in A$ be a Haar
unitary generator of $A$. Let
\begin{align*}
 W(v)=\{\zeta\in L^{2}(M)\ominus L^{2}(A): \mathbb{E}_{A}(\zeta v^{n}\zeta^{*})=0,\text{ for all }n\neq 0\}.
\end{align*}
Then $span \text{ }W(v)$ is dense in $L^{2}(A)^{\perp}$.
\end{Theorem}

\begin{proof}
Note that for $\zeta\in L^{2}(M)\ominus L^{2}(A)$, if $\mathbb{E}_{A}(\zeta v^{n}\zeta^{*})=0$
for all $n\neq 0$ for some Haar unitary generator, then the same is true for any Haar unitary generator 
$($see discussion after Theorem $2.1$ \cite{Muk2}$)$.\\
\indent Let $\mathcal{E}=\left\{\underline{\epsilon}=\{\epsilon_{i,n}\}_{0\leq
i< n, n\in Puk(A)}: \epsilon_{i,n}=\pm 1\right\}$. Without loss of generality, let $v$
correspond to the function $t\mapsto e^{2\pi it}$. The
left--right measure of $A$ is $[\lambda\otimes \lambda]$. For
$\mathbb{N}\cup\{\infty\}\ni n\in Puk(A)$ there exist vectors
$\zeta^{(n)}_{i}$, $0\leq i< n$, so that the projections
$P_{i}^{(n)}:L^{2}(M)\rightarrow 
\overline{A\zeta^{(n)}_{i}A}^{\norm{\cdot}_{2}}$ are mutually
orthogonal, equivalent in $\A^{\prime}$,
$\overline{A\zeta^{(n)}_{i}A}^{\norm{\cdot}_{2}} \perp L^{2}(A)$,
$\A^{\prime}(\sum_{0\leq i<n}P_{i}^{(n)})$ is the type $\rm{I}_{n}$ central summand 
of $\mathcal{A}^{\prime}(1-e_{A})$, and,
for $a,b\in C[0,1]\subset A$ and for all $\underline{\epsilon}\in
\mathcal{E}$,
\begin{align}
\nonumber \langle a \left(\underset{n\in
Puk(A)}\oplus\left({\underset{0\leq i<n}\oplus}\epsilon_{i,n}\zeta^{(n)}_{i}\right)\right)
b,\underset{n\in
Puk(A)}\oplus\left({\underset{0\leq i<n}\oplus}\epsilon_{i,n}\zeta^{(n)}_{i}\right)\rangle
=\int_{[0,1]\times [0,1]}a(t)b(s)d\lambda(t)d\lambda(s).
\end{align}
\indent Fix $\underline{\epsilon}\in \mathcal{E}$ and let
$\zeta_{\underline{\epsilon}}=\underset{n\in
Puk(A)}\oplus\text{ }{\underset{0\leq i<n}\oplus}\epsilon_{i,n}\zeta^{(n)}_{i}$.
By Lemma \ref{identify_disintegrated_measure}, we find
\begin{align}\label{cond_expc_zero}
\norm{\mathbb{E}_{A}(\zeta_{\underline{\epsilon}}
v^{n}\zeta_{\underline{\epsilon}}^{*})}_{1}=\int_{0}^{1}\abs{\lambda(
1\otimes v^{n})}d\lambda(t)=0, \text{ for all }n\neq 0.
\end{align}
Note that
$\zeta_{\underline{\epsilon}}\perp L^{2}(A)$. For $u\in
\mathcal{U}(A)$ and $b,c\in A$, use equation \eqref{cond_expc_zero} to conclude that 
\begin{align}
\nonumber \mathbb{E}_{A}((b\zeta_{\underline{\epsilon}}
u)v^{n}(c\zeta_{\underline{\epsilon}}
u)^{*})=b\mathbb{E}_{A}(\zeta_{\underline{\epsilon}}v^{n}\zeta_{\underline{\epsilon}}^{*}
)c^{*}=0 \text{ for all }n\neq 0.
\end{align}
Let
\begin{align}
\nonumber W=\text{ span }\left\{b\zeta_{\underline{\epsilon}}u :u\in
\mathcal{U}(A), b\in A, \underline{\epsilon}\in \mathcal{E}\right\}.
\end{align}
It is easy to check that $W$ is dense in $L^{2}(A)^{\perp}$.
\end{proof}

\section{Masas from Dynamical systems}

We begin this section with a remark about mixing measures. By
the Riemann--Lebesgue Lemma, any measure absolutely continuous with
respect to the Lebesgue measure is mixing; however, there are many mixing
singular measures as well. Any measure absolutely continuous with
respect to a mixing measure is mixing. Thus, mixing is a property of
equivalence classes of measures. Mixing measures can be characterized
in a geometric way as being asymptotically uniformly distributed
\cite[Proposition 2.6]{Kat}. In this section, we will analyze the notion of mixing masas from a
spectral theory point of view.\\
\indent In the next theorem, we relate  mixing actions of countable
discrete abelian groups to Fourier coefficients of the left--right measures of 
the associated mixing masas in the crossed product constructions. For simplicity, we
work with $\mathbb{Z}$--actions. 

\begin{Theorem}\label{action_sm}
Let $\alpha$ be a free  mixing action of $\mathbb{Z}$ on a
diffuse separable finite von Neumann algebra $N$ preserving a
faithful normal tracial state $\tau$. If $[\eta]$ is the
left--right measure of $L(\mathbb{Z})\subset N\rtimes
_{\alpha}\mathbb{Z}$, then $\tilde{\eta}^{t}$ is a mixing measure
for $\lambda$ almost all $t$.
\end{Theorem}

\begin{proof}
Let $M=N\rtimes _{\alpha}\mathbb{Z}$. The tracial state on $M$ will
be denoted by $\tau$ as well. 
Let $u_{n}\in M$ be the canonical unitaries implementing the action. 
Suppose $x\in N$ and $n,n_{1},n_{2}\in
\mathbb{Z}$. Then, the equation
\begin{align}\label{trace1}
\langle
u_{n_{1}}xu_{n}u_{n_{2}},xu_{n}\rangle&=\tau(u_{n_{1}}xu_{n}u_{n_{2}}u_{-n}x^{*})
=\tau(u_{n_{1}}xu_{n_{2}}x^{*}),
\end{align}
implies that $\eta_{xu_{n}}=\eta_{x}$ for all $x\in N$ and all $n\in
\mathbb{Z}$. \\
\indent Note that the left--right measure of
$L(\mathbb{Z})\subset M$ is naturally supported on
$\widehat{\mathbb{Z}}=S^{1}$. Identify
$L(\mathbb{Z})=L^{\infty}(S^{1},\lambda)$, where $\lambda$ is the normalized
Haar measure on $S^{1}$, via the standard identification which
sends $u_{n}$ to the function $e_{n}(t)=t^{n}$, $t\in S^{1}$, $n\in
\mathbb{Z}$. Now, for $x\in N$ and $m\in \mathbb{Z}$,
\begin{align*}
\mathbb{E}_{L(\mathbb{Z})}(xu_{m}x^{*})&=\mathbb{E}_{L(\mathbb{Z})}(x\alpha_{m}(x^{*})u_{m})
=\mathbb{E}_{L(\mathbb{Z})}(x\alpha_{m}(x^{*}))u_{m}
=\tau(x\alpha_{m}(x^{*}))u_{m}.
\end{align*}
\indent Therefore, from Lemma \ref{identify_disintegrated_measure}, $\eta_{x}^{t}(1\otimes
e_{m})=\tau(x\alpha_{m}(x^{*}))e_{m}(t)$ for $\lambda$ almost all $t\in
S^{1}$ and for all $m$. Since the action $\alpha$ is mixing, so
$\tilde{\eta}_{x}^{t}$ is a mixing measure for $\lambda$ almost all $t$
whenever $\tau(x)=0$.\\
\indent Let $x=\sum_{i=1}^{n}x_{i}u_{k_{i}}\in M$ be such that
$\mathbb{E}_{L(\mathbb{Z})}(x)=0$ and $k_{i}\neq k_{j}$ for $i\neq j$. Therefore, 
$\tau(x_{i})=0$, for all $1\leq i\leq n$. Now from equation \eqref{trace1}, we get
$\eta_{x}=\sum_{i=1}^{n}\eta_{x_{i}}+\sum_{i\neq
j=1}^{n}\eta_{x_{i}u_{k_{i}}, x_{j}u_{k_{j}}}$. It is easy to see
that $d\eta_{x_{i}u_{k_{i}}, x_{j}u_{k_{j}}}=(1\otimes
e_{k_{i}-k_{j}})d\eta_{x_{i},x_{j}}$ for all $i\neq j$. Thus, from
equation \eqref{polarization1} and Lemma 3.6 \cite{Muk1} it follows that,
\begin{align}
\nonumber \int_{S^{1}}s^{m}d\tilde{\eta}_{x_{i}u_{k_{i}},
x_{j}u_{k_{j}}}^{t}(s)=\int_{S^{1}}s^{m}s^{k_{i}-k_{j}}d\tilde{\eta}_{x_{i},x_{j}}^{t}(s)\rightarrow
0 \text{ as }m\rightarrow \infty
\end{align}
for $\lambda$ almost all $t$. This shows that $\tilde{\eta}_{x}^{t}$ is
a mixing measure for $\lambda$ almost all $t$.\\
\indent There is a unit vector $\zeta\in L^{2}(M)\ominus
L^{2}(L(\mathbb{Z}))$ such that $\eta=\eta_{\zeta}$. Let $
x_{n}=\sum_{i=1}^{k_{n}}x_{i}^{(n)}u_{k_{i}}^{(n)}\in M$ with
$x_{i}^{(n)}\in N$ be such that $\mathbb{E}_{A}(x_{n})=0$,
$\norm{x_{n}}_{2}\leq 1$ for all $n\in \mathbb{N}$ and
$x_{n}\rightarrow \zeta $ as $n\rightarrow \infty$ in
$\norm{\cdot}_{2}$. Then, $\eta_{x_{n}}\rightarrow \eta_{\zeta}=\eta$
in $\norm{\cdot}_{t.v}$ from Lemma 3.10 \cite{Muk1}. Then from Lemma
3.9 \cite{Muk1}, there is a subsequence $n_{k}$ with $n_{k}<n_{k+1}$
for all $k$ and a set $E\subset S^{1}$ with $\lambda(E)=0$, such that
for all $t\in E^{c}$, $\tilde{\eta}_{x_{n_{k}}}^{t},\tilde{\eta}^{t}$ are finite and
$$\underset{A\subseteq S^{1},A\text{ Borel}}\sup
\abs{\tilde{\eta}_{x_{n_{k}}}^{t}(A)-\tilde{\eta}^{t}(A)}\rightarrow
0 \text{ as }k\rightarrow \infty.$$ Note that
$\tilde{\eta}_{x_{n_{k}}}^{t}$ are mixing measures for all $k$ and
for $\lambda$ almost all $t$. From standard approximation arguments, it
follows that $\tilde{\eta}^{t}$ is a mixing measure for $\lambda$ almost all $t$.
\end{proof}

We now study relations between the spectral measure of an
action and the left--right measure of a masa that arises from
a dynamical system. 
Before doing so, we need some preparation on unitary representations. Let
$H$ be a locally compact abelian $($LCA$)$ group. Note that
$\widehat{H}$ is also a LCA group, where $\widehat{H}$ denotes the
Pontryagin dual of $H$. Also note that, if $H$ is discrete, then
$\widehat{H}$ is compact and vice versa. Let $\sigma_{\widehat{H}}$
denote the Borel $\sigma$--algebra of $\widehat{H}$. The following
result was proved by Stone for the case $H=\mathbb{R}$ and then
independently generalized by Naimark, Ambrose and Godement, and is
called the SNAG theorem \cite[Theorem D.3.1]{(T)}.

\begin{Theorem}[SNAG Theorem]\label{snag}
Let $(\pi,\h)$ be a strongly continuous unitary representation of a
LCA group $H$ on a separable Hilbert space $\h$. Then there exists
a unique regular projection valued measure
$E_{\pi}:\sigma_{\widehat{H}}\rightarrow Proj(\h)$ on $\widehat{H}$ such
that
\begin{align}
\nonumber \pi(g)=\int_{\widehat{H}}\overline{\chi(g)}dE_{\pi}(\chi),
\text{   for all }g\in H.
\end{align}
\end{Theorem}

\indent Let $(X,\nu)$ be a Lebesgue probability space, where $X$ is
a compact metrizable space. Let $H$ be a countable discrete abelian
group and let $\alpha$ be an automorphic $($free$)$ ergodic action of
$H$ on $X$ preserving the measure $\nu$. This gives rise to a
canonical unitary representation $\pi: H\rightarrow
\mathbf{B}(L^{2}(X,\nu)\ominus \mathbb{C}1)$. Let $\mu$ be the
maximal spectral type of this representation that arises from a
vector $f_{0}\in L^{2}(X,\nu)\ominus \mathbb{C}1$ \cite[p. 13]{Nad}. 
Consequently, by SNAG theorem and
Hahn--Hellinger theorem \cite{Nad}, there is a $\mu$--measurable field
of Hilbert spaces $\{\h_{\psi}\}_{\psi\in \widehat{H}}$ such that
\begin{align}\label{di}
L^{2}(X,\nu)\ominus \mathbb{C}1\cong
\int_{\widehat{H}}^{\oplus}\h_{\psi}d\mu(\psi)
\end{align}
and $L^{\infty}(\widehat{H},\mu)$ is unitarily equivalent to the
algebra of diagonalizable operators with respect to the
decomposition in equation \eqref{di}. For $f,g\in L^{2}(X,\nu)\ominus
\mathbb{C}1$, denote by $\mu_{f,g}$ the $($possibly$)$ complex Borel
measure on $\widehat{H}$ obtained as $\mu_{f,g}(B)=\langle
E_{\pi}(B)f,g\rangle$, where $B\subseteq \widehat{H}$ is Borel. We
will also denote $\mu_{f,f}$ by $\mu_{f}$. Thus $\mu_{f_{0}}=\mu$.
The dimension function of this decomposition in equation \eqref{di} is
the \emph{spectral multiplicity} of the representation $\pi$. 
Now
$($see Proposition 11 p. 174, \cite{Dix}$)$
\begin{align}\label{doubledint}
(L^{2}(X,\nu)\ominus
\mathbb{C}1)\otimes L^{2}(\widehat{H},\lambda_{\widehat{H}})\cong \int_{\widehat{H}\times
\widehat{H}}^{\oplus}\h_{\psi,\chi}d\mu(\psi)d\lambda_{\widehat{H}}(\chi),
\text{ where }\h_{\psi,\chi}=\h_{\psi},
\end{align}
and $\lambda_{\widehat{H}}$ is the normalized Haar measure on $\widehat{H}$.\\
\indent Each $h\in H$ defines a continuous function
$\widehat{h}:\widehat{H}\rightarrow \mathbb{C}$ by
$\widehat{h}(\chi)=\chi(h)$, $\chi\in \widehat{H}$. Furthermore,
$h\in H$ defines a unitary operator $m_{\widehat{h}}$ on
$L^{2}(\widehat{H},\lambda_{\widehat{H}})$ given by
$m_{\widehat{h}}(f)=\widehat{h}f$, $f\in
L^{2}(\widehat{H},\lambda_{\widehat{H}})$, and a projection
$e_{\widehat{h}}$ projecting onto $\mathbb{C} \widehat{h}$. 
Via the Fourier transform, the crossed product factor
$R_{\alpha}=L^{\infty}(X,\nu)\rtimes_{\alpha}H$ $($which of course
is the hyperfinite $\rm{II}_{1}$ factor$)$ is generated by $\left\{\sum_{h\in
H}\alpha_{h}(f)\otimes e_{\widehat{h}}: f\in L^{\infty}(X,\nu), h\in
H\right\}$ and $\left\{1\otimes m_{\widehat{h^{-1}}}:h\in H\right\}$
on $L^{2}(X,\nu)\otimes L^{2}(\widehat{H},\lambda_{\widehat{H}})$.
Note that we follow \cite{(T)} for the definition of Fourier transform.
One considers this standard $($GNS$)$ representation of
$R_{\alpha}$ on $L^{2}(X,\nu)\otimes L^{2}(\widehat{H},\lambda_{\widehat{H}})$. 
Let $J$ denote the
conjugation operator on the space $L^{2}(X,\nu)\otimes
L^{2}(\widehat{H},\lambda_{\widehat{H}})$. For $g\in H$, let $v_{g}$
be the unitary in $\mathbf{B}(L^{2}(X,\nu))$ that implements the
automorphism $\alpha_{g}$. That is, if $T_{g}$ is the measure
preserving transformation such that $\alpha_{g}(f)=f\circ
T_{g}^{-1}$, $f\in L^{\infty}(X,\nu)$, then $v_{g}a=a\circ
T_{g}^{-1}$ for all $a\in L^{2}(X,\nu)$. It is easy to see that for
$f\in L^{\infty}(X,\nu)$ and $g\in H$,
\begin{align}\label{eq1}
&J\left((1\otimes m_{\widehat{g}})(\sum_{h\in H}\alpha_{h}(\overline{f})\otimes e_{\widehat{h}})\right)=(v_{g}\otimes m_{\widehat{g^{-1}}})\left(\sum_{h\in H}\alpha_{h}(f)\otimes e_{\widehat{h}}\right), \text{ and}\\
\nonumber &J(1\otimes m_{\widehat{g}})J=(v_{g}\otimes
m_{\widehat{g^{-1}}}).
\end{align}

\begin{Theorem}\label{left--right}
The left--right measure $[\eta_{\mid\Delta(\widehat{H})^{c}}]$ of the
masa $L(H)\subset R_{\alpha}$ is the equivalence class of
$S_{*}(\mu\otimes \lambda_{\widehat{H}})$, where $\mu$ is the maximal
spectral type of the unitary representation of $H$ on
$L^{2}(X,\nu)\ominus \mathbb{C}1$ that arises from the action
$\alpha$ and $S:\widehat{H}\times \widehat{H}\rightarrow
\widehat{H}\times \widehat{H}$ is given by
$S(\psi,\chi)=(\chi,\chi\psi)$.
\end{Theorem}

\begin{proof}
Let $\tau_{\mathcal{R}_{\alpha}}$ denote the faithful normal tracial state of 
$\mathcal{R}_{\alpha}$. It is clear that $\widehat{H}\times \widehat{H}$ is the natural
space where the left--right measure is to be built. Write
$\eta_{\mid\Delta(\widehat{H})^{c}}=\eta$. For $f\in
L^{\infty}(X,\nu)$ and $g\in H$, write $\alpha(f)=\sum_{h\in
H}\alpha_{h}(f)\otimes e_{\widehat{h}}$ and $w_{g}=1\otimes
m_{\widehat{g}}$. The operator $w_{g}$ is canonically identified with the 
function $\widehat{g}\in C(\widehat{H})$. For $i=1,2$, fix $f_{i}\in
L^{\infty}(X,\nu)$ and $h_{i}\in H$. Now for $g_{1},g_{2}\in H$,
\begin{align}\label{maj}
\langle
w_{g_{1}}\alpha({f}_{1})w_{h_{1}}w_{g_{2}},\text{ }\alpha({f}_{2})w_{h_{2}}\rangle_{\tau_{R_{\alpha}}}&=\int_{\widehat{H}\times\widehat{H}}\widehat{g_{1}}(\psi)\widehat{g_{2}}(\chi)\widehat{h_{1}h_{2}^{-1}}(\chi)d\eta_{\alpha({f}_{1}),\alpha({f}_{2})}(\psi,\chi).
\end{align}
On the other hand,
\begin{align}\label{major1}
&\langle
w_{g_{1}}\alpha({f}_{1})w_{h_{1}}w_{g_{2}},\text{ }\alpha({f}_{2})w_{h_{2}}\rangle_{\tau_{R_{\alpha}}}\\
\nonumber=& \langle w_{g_{1}} Jw_{g_{2}^{-1}}J(\alpha({f}_{1})w_{h_{1}}),\text{ }(\alpha({f}_{2})w_{h_{2}})\rangle_{\tau_{R_{\alpha}}} \\
\nonumber =&\langle (1\otimes m_{\widehat{g_{1}g_{2}}})\alpha(v_{g_{2}^{-1}}f_{1})w_{h_{1}},\text{ }\alpha(f_{2})w_{h_{2}}\rangle_{\tau_{R_{\alpha}}} \text{ (by equation  }\eqref{eq1})\\
\nonumber =&\tau_{R_{\alpha}}\left(w_{g_{1}g_{2}}\alpha(v_{g_{2}^{-1}}f_{1})w_{h_{1}h_{2}^{-1}}(\alpha(f_{2}))^{*}\right) \\
\nonumber=&\tau_{R_{\alpha}}\left(w_{g_{1}g_{2}h_{1}h_{2}^{-1}}\alpha(v_{g_{2}^{-1}h_{1}h_{2}^{-1}}(f_{1}))(\alpha({f}_{2}))^{*}\right)\\
\nonumber=&\tau_{R_{\alpha}}(w_{g_{1}g_{2}h_{1}h_{2}^{-1}})\langle v_{g_{2}^{-1}h_{1}h_{2}^{-1}}f_{1},f_{2}\rangle_{L^{2}(X,\nu)} \text{ (by orthogonality of algebras)}\\
\nonumber=&\int_{\widehat{H}}\chi(g_{1}g_{2})\chi(h_{1}h_{2}^{-1})d\lambda_{\widehat{H}}(\chi)\int_{\widehat{H}}\overline{\psi(g_{2}^{-1}h_{1}h_{2}^{-1})}d\mu_{f_{1},f_{2}}(\psi) \text{ (by SNAG Theorem)}\\
\nonumber=&\int_{\widehat{H}\times \widehat{H}}\text{
}\psi(g_{2})\overline{\psi(h_{1}h_{2}^{-1})}\chi(g_{1})\chi(g_{2})\chi(h_{1}h_{2}^{-1})d\mu_{f_{1},f_{2}}(\psi)d\lambda_{\widehat{H}}(\chi).
\end{align}
\indent Let $S:\widehat{H}\times \widehat{H}\rightarrow \widehat{H}\times \widehat{H}$ be given by
$S(\psi,\chi)=(\chi,\chi\psi)$. Note that $S$ is bijective. As
discussed before, let $f_{0}\in L^{2}(X,\nu)\ominus \mathbb{C}1$ be
the vector such that $\mu_{f_{0}}=\mu$. Now for $f\in
L^{\infty}(X,\nu)$, from \eqref{maj} and \eqref{major1}, we have
\begin{align}
\nonumber (\mu_{f}\otimes\lambda_{\widehat{H}})\circ
S^{-1}=\eta_{\alpha(f)}.
\end{align}
\indent Let $n\in \mathbb{N}$ and $f_{i}\in L^{\infty}(X,\nu)$, $h_{i}\in H$
for $1\leq i\leq n$ be such that
$\mathbb{E}_{L(H)}(\sum_{i=1}^{n}\alpha(f_{i})w_{h_{i}})=0$. Then, 
$\int_{X}f_{i}d\nu=0$ and from equation \eqref{maj} it follows that
\begin{align}
\nonumber
d\eta_{\sum_{i=1}^{n}\alpha({f}_{i})w_{h_{i}}}=\sum_{i,j=1}^{n}(1\otimes
\widehat{h_{i}h_{j}^{-1}})d\eta_{\alpha(f_{i}),\alpha(f_{j})}.
\end{align}
Thus, $\eta_{\sum_{i=1}^{n}\alpha({f}_{i})w_{h_{i}}}\ll
\eta_{\alpha(\sum_{i=1}^{n}{f}_{i})}=(
\mu_{\sum_{i=1}^{n}{f}_{i}}\otimes\lambda_{\widehat{H}})\circ S^{-1}\ll
(\mu_{f_{0}}\otimes\lambda_{\widehat{H}})\circ S^{-1}$. Note that
$((\mu_{f_{0}}\otimes\lambda_{\widehat{H}})\circ
S^{-1})(\Delta(\widehat{H}))=0$. There is a nonzero vector $\zeta\in
L^{2}(R_{\alpha})\ominus L^{2}(\widehat{H},\lambda_{\widehat{H}})$
such that $\eta_{\zeta}=\eta$. By an easy approximation argument it
follows that $\eta\ll (\mu_{f_{0}}\otimes\lambda_{\widehat{H}})\circ S^{-1}$ 
$($\cite[Lemma $3.10$]{Muk1}$)$. Note that equations \eqref{maj}, \eqref{major1}
extend to functions $f_{i}\in L^{2}(X,\nu)$, in particular these equations are valid
for $f_{1}=f_{2}=f_{0}$. Working similarly with $f_{1}=f_{0}$ and
$f_{2}=f_{0}$ in equations \eqref{maj}, \eqref{major1}, one checks that
$(\mu_{f_{0}}\otimes\lambda_{\widehat{H}})\circ
S^{-1}=\eta_{f_{0}\otimes 1}$. Thus, from Lemma 5.7 \cite{DSS}
conclude that $[(\mu_{f_{0}}\otimes\lambda_{\widehat{H}})\circ
S^{-1}]=[\eta]$.\\
\indent Finally, from equation \eqref{doubledint},
\begin{align}\label{calpuk}
L^{2}(R_{\alpha})\ominus
L^{2}(\widehat{H},\lambda_{\widehat{H}})\cong
\int_{\widehat{H}\times
\widehat{H}}^{\oplus}\h_{\chi\psi^{-1}}dS_{*}(\mu\otimes\lambda_{\widehat{H}})(\psi,\chi)
\end{align}
and $(L(H)\cup JL(H)J)^{\prime\prime}(1-e_{L(H)})$ is diagonalizable with
respect to this decomposition.
\end{proof}

\begin{Remark}
\emph{Theorem \ref{left--right} will be used in the next section to distinguish 
mixing masas in the free group factors}.
\end{Remark}

\begin{Remark}\label{DynamNoexist}
\emph{It is a long standing open question in ergodic theory that, 
whether there exists a measure preserving automorphism of a Lebesgue probability 
space whose maximal spectral type is absolutely continuous but not Lebesgue. 
There are philosophies that suggest that the answer could go either way. 
For an excellent account of Koopman--realizable $($through a $\mathbb{Z}$--action$)$ 
measures and multiplicities check \cite{KatL}. Observe that Theorem \ref{left--right}, 
Theorem \ref{inclusion_of_subgroups} and Corollary \ref{notexist} say that such 
a Koopman--realizable measure does not exist provided we restrict ourselves to a 
smaller class of dynamical systems, namely, those that arise as semidirect 
products of groups.}
\end{Remark}

\begin{Corollary}\label{formal}
Let $\Gamma$ be any countable discrete group such that $\Gamma_{0}\leq Aut(\Gamma)$,  
where $\Gamma_{0}$ is a countable discrete torsion-free abelian group. Let 
the canonical action of $\Gamma_{0}$ on $L(\Gamma)$ be mixing. Then the maximal spectral 
type of the $\Gamma_{0}$--action is Lebesgue.
\end{Corollary}

\section{Mixing masas in the free group factors}

The understanding of singular masas especially in the free group factors is 
of worth in the subject. For construction of masas in this section, we require substantial techniques 
from ergodic theory. In this section, we exhibit uncountably many  non conjugate mixing
masas in the free group factors with Puk\'{a}nszky invariant $\{1,\infty\}$.
The general strategy is to construct suitable masas in finite amenable
von Neumann algebras and appeal to a well known result of Dykema regarding free
products \cite{Dy}.\\
\indent Recall that the rank of a measure preserving
automorphism on a standard probability space is greater than or
equal to the spectral multiplicity of the associated Koopman
operator \cite[p. 31]{Nad}. Thus, if the rank of a mixing automorphism is
one, then the spectral multiplicity of the associated Koopman
operator must also  be one. In the previous section, we have related the
spectral multiplicity of a transformation to the Puk\'{a}nszky invariant of the 
associated masa $($along the direction of the group$)$ in the group measure space construction.\\
\indent A rank one measure preserving transformation $T$ of the unit interval $[0,1]$
is constructed by the method of cutting and stacking
\cite{Fri,K-R,Nad}. These are transformations which admit a sequence of \emph{Rokhlin towers} generating the entire 
$\sigma$--algebra. We will assume that the reader is familiar
with the notion of cutting and stacking. The \emph{classical staircase
transformation} is one in which, at the $k$--th stage, one divides the
$(k-1)$--th stack into $k$ equal columns and put $j$ spacers over the
$j$--th column, $1\leq j\leq k$, which is why it is called a
staircase \cite[p. 153]{Nad}. The next result will be used in constructing 
masas in the free group factors, but it is also of some independent 
interest as well. 

\begin{Theorem}\label{SingHyperfinite}
There exists a mixing masa $B$ in the hyperfinite $\rm{II}_{1}$ factor $R$ whose Puk\'{a}nszky invariant
is $\{1\}$ and whose left--right measure is singular.
\end{Theorem}

\begin{proof}
The classical staircase automorphism $T$ is mixing and of rank one \cite[p. 744]{Ad}.
Consequently, $L(\mathbb{Z})\subset R_{T}$ is a  mixing masa in the hyperfinite $\rm{II}_{1}$ factor
$R_{T}=L^{\infty}([0,1],\lambda)\rtimes_{T}\mathbb{Z}$, where $\lambda$ is the Lebesgue measure on $[0,1]$ \cite{J-S}. The Puk\'{a}nszky invariant of $L(\mathbb{Z})\subset R_{T}$ is $\{1\}$ from equation
\eqref{calpuk}. Write $B=L(\mathbb{Z})$.

The maximal spectral type of the staircase transformation is given by a Reisz product, which is
known to be singular \cite{Kle} $($also see p. 154 \cite{Nad}$)$. Thus from Theorem~\ref{left--right}, the
left--right measure of $B\subset R_{T}$ is singular to the
product class.
\end{proof}

\begin{Theorem}\label{arbitrary}
Let $k\in \left\{2,3,\ldots,\infty\right\}$ and let $\Gamma$ be any
countable discrete group. There exist uncountably many pairwise non conjugate
mixing masas in $L(\mathbb{F}_{k}*\Gamma)$ whose
Puk\'{a}nszky invariant is $\{1,\infty\}$.
\end{Theorem}

\begin{proof}
Let
\begin{align}
\mathbb{P}_{\mathbb{N}}=\left\{
\alpha=\{\alpha_{n}\}_{n=1}^{\infty}: \alpha_{n}>\alpha_{n+1},
0<\alpha_{n}<1 \text{ for all }
n,\sum_{n=1}^{\infty}\alpha_{n}=1\right\}.
\end{align}
For $\alpha,\beta \in \mathbb{P}_{\mathbb{N}}$, say $\alpha\neq
\beta$ if $\alpha_{n}\neq \beta_{n}$ for some $n$. Fix $\alpha\in
\mathbb{P}_{\mathbb{N}}$.\\
\indent  Let $R_{\alpha}=\oplus_{n=1}^{\infty}R$ and $B_{\alpha}=\oplus_{n=1}^{\infty}B$, where
$R$ and $B$ are as in Theorem \ref{SingHyperfinite}. The projections 
$p_{n}=(0\oplus \cdots \oplus 0\oplus 1\oplus
0\oplus \cdots)$, where $1$ appears at the $n$--th coordinate, is a
central projection of $R_{\alpha}$ and it belongs to
$B_{\alpha}$. Equip $R_{\alpha}$ with the
faithful trace
\begin{align*}
\tau_{R_{\alpha}}(\cdot)=\sum_{n=1}^{\infty}\alpha_{n}\tau_{R}(\cdot p_{n}),
\end{align*}
where $\tau_{R}$ denotes the unique tracial state of $R$. 
Then $B_{\alpha}$ is a mixing masa in the
hyperfinite algebra $R_{\alpha}$ and the latter is separable. The last
statement is a simple application of dominated convergence theorem. \\
\indent  The projections $p_{n}$ correspond to indicator of
measurable subsets $E_{n}\subset (S^{1},\textbf{m})$, so that $E_{n}\cap
E_{m}$ is a set of $\textbf{m}$ measure $0$ for all $n\neq m$ 
$($where $\textbf{m}$ is the normalized Haar measure on $S^{1})$. Upon
applying appropriate transformations, the left--right measure
of $B\subset R$ can be transported to each $E_{n}\times E_{n}$,
which is denoted by $[\eta_{n}]$. We also assume
$\eta_{n}(E_{n}\times E_{n})=1$ for all $n$.\\
\indent Consider
$(M,\tau_{M})=(R_{\alpha},\tau_{R_{\alpha}})*(R,\tau_{R})$. Then $M$
is isomorphic to $L(\mathbb{F}_{2})$ by a well known theorem of
Dykema \cite{Dy}. Note that $B_{\alpha}\subset L(\mathbb{F}_{2})$ is a
mixing masa by Propositions 6.1, 6.5 \cite{CFM}. 
The left--right measure
of the inclusion $B_{\alpha}\subset L(\mathbb{F}_{2})$ is
\begin{align}
\nonumber[\textbf{m}\otimes\textbf{m} +
\sum_{n=1}^{\infty}\frac{1}{2^{n}}\eta_{n}]
\end{align}
and $Puk_{L(\mathbb{F}_{2})}(B_{\alpha})=\{1,\infty\}$ from
Proposition $5.10$ and Theorem $3.2$ \cite{DSS} $($also see \cite{Muk2}$)$.\\
\indent Since automorphisms of $\rm{II}_{1}$ factors are trace
preserving, the non conjugacy of $B_{\alpha}$ and $B_{\beta}$ in
$L(\mathbb{F}_{2})$ for $\alpha\neq \beta$ follows clearly by
considering their left--right measures in the measure--multiplicity invariant.\\
\indent There exist isomorphisms \cite{Dy}
\begin{align}
\nonumber & L(\mathbb{F}_{2})*L(\mathbb{F}_{k-2}*\Gamma)\cong
L(\mathbb{F}_{k}*\Gamma) \text{ for }k\geq 2.
\end{align}
For $k\geq 2$, each $B_{\alpha}$ is a mixing masa $($Propositions 6.1, 6.5 \cite{CFM}$)$ in
$L(\mathbb{F}_{k}*\Gamma)$ with
$Puk_{L(\mathbb{F}_{k}*\Gamma)}(B_{\alpha})=\{1,\infty\}$
\cite{DSS}. Use Lemma 5.7, Proposition 5.10 \cite{DSS} to decide the non
conjugacy of $B_{\alpha}$ and $B_{\beta}$ when $\alpha\neq \beta$ in
the free product.
\end{proof}

\begin{Remark}
\emph{It is difficult to distinguish between two mixing masas in the free group factors. 
If two masas in a free group factor are of product class, then it becomes a significantly 
difficult problem $($e.g., the conjugacy of the Laplacian masa and the generator masas is a 
challenging problem \cite{DyMu}$)$. Maximal injectivity of masas can be used but that too in very 
limited cases. The left--right measure of any masa in the free group factors always 
contains a part of the product measure as a summand \cite{Voi}. This statement 
is one of the deep results in the subject. So, the singular summand of the 
left--right measure is a plausible candidate that can distinguish two masas 
with same Puk\'{a}nszky invariant. It is a very common idea to build a masa in 
a free group factor by starting with a masa in the hyperfinite $\rm{II}_{1}$ factor. 
But mixing masas in the hyperfinite $\rm{II}_{1}$ factor arising from ergodic group actions
are also rare. The set of mixing transformations is meager in the $($Polish$)$ weak 
topology on the group of all measure preserving transformations. In \cite{Ti}, a 
Polish topology strictly stronger than the induced weak topology was introduced 
and it was shown that a generic mixing transformation has multiplicity $\{1\}$. 
Nevertheless, many more values of the multiplicity function of mixing transformations 
were obtained by Danilenko in \cite{Da}. But it is yet not known whether the maximal 
spectral types of these transformations in \cite{Da} are singular to Lebesgue measure. 
In case they are, our technique applies to construct more mixing masas in 
the free group factors.}
\end{Remark}

\appendix
\section{A Technical Result}

As before, let $A$ be a masa in a $\rm{II}_{1}$ factor $M$. We continue to assume 
that $A\cong L^{\infty}([0,1],\lambda)$, where $\lambda$ is the Lebesgue measure. 
The next Lemma was remarked in \cite{Muk2} under a stronger hypothesis. We assume 
the theory of $L^{1}$ spaces associated to finite von Neumann algebras for which 
we refer the reader to \cite{S-S2}.

\begin{Lemma}\label{Lemma:appendix}
Let the left--right measure of
$A$ be the class of product measure. Let $v\in A$ be the Haar unitary generator corresponding to the function $[0,1]\ni t\mapsto e^{2\pi it}$. Then the following are equivalent.\\
\noindent $(i)$ There is a set $S\subset
L^{2}(M)\ominus L^{2}(A)$ such that $\text{ span }S$ is
dense in $L^{2}(M)\ominus L^{2}(A)$, and 
\begin{align}
\nonumber \sum_{k\in \mathbb{Z}}\norm{\mathbb{E}_{A}(\zeta
v^{k}\zeta^{*})}_{2}^{2}<\infty \text{ for all }\zeta\in S.
\end{align}
\noindent $(ii)$ There is a set $S^{\prime}\subset L^{2}(M)\ominus L^{2}(A)$ such that 
$S^{\prime}$ is dense in $L^{2}(M)\ominus L^{2}(A)$, and
\begin{align}
\nonumber \sum_{k\in \mathbb{Z}}\norm{\mathbb{E}_{A}(\zeta_{1}
v^{k}\zeta_{2}^{*})}_{2}^{2}<\infty \text{ for all }\zeta_{1},\zeta_{2}\in S^{\prime}.
\end{align}
\end{Lemma}

The vectors $\mathbb{E}_{A}(\zeta_{1}v^{k}\zeta_{2}^{*}), \mathbb{E}_{A}(\zeta
v^{k}\zeta^{*})$ in the statement of the above lemma 
are in $L^{1}(M)$. But in the statement of Lemma \ref{Lemma:appendix}, it is implicit that the sets
$S, S^{\prime}$ can be chosen so that $\mathbb{E}_{A}(\zeta_{1}v^{k}\zeta_{2}^{*}),\mathbb{E}_{A}(\zeta
v^{k}\zeta^{*})\in L^{2}(A)$. Thus, there is no
confusion in considering their $L^{2}$--norms.

\begin{proof}
We have to prove $(i)\Rightarrow (ii)$ only. For $b\in C[0,1]$ and $\zeta\in S$, 
the Fourier series expansion of $b$ and a simple application of 
Cauchy--Schwarz inequality show that $\mathbb{E}_{A}(\zeta b\zeta^{*})\in L^{2}(A)$. 
Indeed, $b=\sum_{k\in \mathbb{Z}}\langle b,v^{k}\rangle v^{k}$ with convergence in 
$\norm{\cdot}_{2}$. Thus, 
\begin{align}
\nonumber \underset{a\in A:\norm{a}_{2}\leq 1}\sup \abs{\tau(\mathbb{E}_{A}(\zeta b\zeta^{*})a)}&\leq 
\underset{a\in A:\norm{a}_{2}\leq 1}\sup \sum_{k\in \mathbb{Z}} \abs{\tau(\mathbb{E}_{A}(\zeta v^{k}\zeta^{*})a)}\abs{\tau(bv^{-k})}\\
\nonumber &\leq \sum_{k\in \mathbb{Z}} \norm{\mathbb{E}_{A}(\zeta v^{k}\zeta^{*})}_{2}\abs{\tau(bv^{-k})}\\
\nonumber &\leq \left(\sum_{k\in \mathbb{Z}} \norm{\mathbb{E}_{A}(\zeta v^{k}\zeta^{*})}_{2}^{2}\right)^{\frac{1}{2}}\norm{b}_{2}<\infty.
\end{align}
This proves the claim.\\
\indent From Lemma 5.7 \cite{DSS}, we have $\eta_{\zeta}\ll\lambda\otimes\lambda$. Then  
Lemma \ref{identify_disintegrated_measure} and standard theory of Fourier series show that:\\
$(a)$ $f_{\zeta}=\frac{d\eta_{\zeta}}{d(\lambda\otimes \lambda)}$ is in $L^{2}(\lambda\otimes\lambda)$ for all $\zeta\in S$,\\
$(b)$ $\norm{\mathbb{E}_{A}(\zeta b\zeta^{*})}_{2}^{2}=\int_{0}^{1}\abs{\lambda(f_{\zeta}(t,\cdot)b)}^{2}d\lambda(t)$ for all $b\in C[0,1]$.\\
\indent To prove $(a)$ observe that 
\begin{align}
\nonumber \int_{0}^{1}\sum_{k\in \mathbb{Z}}\abs{\eta_{\zeta}^{t}(1\otimes v^{k})}^{2}d\lambda(t)=\sum_{k\in \mathbb{Z}} \int_{0}^{1}\abs{\eta_{\zeta}^{t}(1\otimes v^{k})}^{2}d\lambda(t)=\sum_{k\in \mathbb{Z}}\norm{\mathbb{E}_{A}(\zeta
v^{k}\zeta^{*})}_{2}^{2}<\infty.
\end{align}
Thus, for $\lambda$ almost all $t$ one has,
\begin{align}
\nonumber \sum_{k\in \mathbb{Z}}\abs{\int_{0}^{1}f_{\zeta}(t,s)v^{k}(s)d\lambda(s)}^{2}<\infty.
\end{align}
Thus $(a)$ is established upon using Lemma 3.6 \cite{Muk1}. The proof of $(b)$ is an easy 
consequence of $(a)$ and Theorem \ref{identify_disintegrated_measure}.\\
\indent Let $\xi_{1}=\sum_{i=1}^{n}c_{i
}\zeta_{i}^{1}$, $\xi_{2}=\sum_{j=1}^{m}d_{j }\zeta^{2}_{j}$ with
$\zeta_{i}^{1},\zeta^{2}_{j}\in S$, $c_{i},d_{j}\in \mathbb{C}$ for all $1\leq i\leq
n$, $1\leq j\leq m$. Note that for all $i,j$,
\begin{align}
\nonumber \sum_{k\in
\mathbb{Z}}\norm{\mathbb{E}_{A}(\zeta_{i}^{1}v^{k}{\zeta_{i}^{1}}^{*})}_{2}^{2}<\infty,
\sum_{k\in
\mathbb{Z}}\norm{\mathbb{E}_{A}(\zeta_{j}^{2}v^{k}{\zeta_{j}^{2}}^{*})}_{2}^{2}<\infty.
\end{align}
Then by assumption $f_{\zeta_{i}^{1}},f_{\zeta_{j}^{2}}\in
L^{2}(\lambda\otimes \lambda)$ for all $i,j$. From equation \eqref{abscont11}, $\eta_{\zeta_{i}^{1},\zeta_{j}^{2}}\ll \lambda\otimes\lambda$. But because $\abs{\eta_{\zeta_{i}^{1},\zeta_{j}^{2}}}\leq
\eta_{\zeta_{i}^{1}}+\eta_{\zeta_{j}^{2}}$, we conclude that
$f_{\xi_{1},\xi_{2}}=\frac{d\eta_{\xi_{1},\xi_{2}}}{d(\lambda\otimes
\lambda)}\in L^{2}(\lambda\otimes \lambda)$.\\
\indent Fix $a,b\in C[0,1]$.
Then as $\tau$ extends to $L^{1}$,
\begin{align}\label{norm_cal}
\int_{0}^{1}a(t)\mathbb{E}_{A}(\xi_{1} b\xi_{2}^{*})(t)d\lambda(t)&=\tau(a\mathbb{E}_{A}(\xi_{1} b\xi_{2}^{*}))=\tau(a\xi_{1} b\xi_{2}^{*})=\int_{[0,1]\times[0,1]}a(t)b(s)d\eta_{\xi_{1},\xi_{2}}(t,s)\\
\nonumber&=\int_{[0,1]\times[0,1]}a(t)b(s)f_{\xi_{1},\xi_{2}}(t,s)d\lambda(t)d\lambda(s)\\
\nonumber&=\int_{0}^{1}a(t)\lambda(f_{\xi_{1},\xi_{2}}(t,\cdot)b)d\lambda(t).
\end{align}
Now consider the function $[0,1]\ni t\overset{g}\mapsto
\lambda(f_{\xi_{1},\xi_{2}}(t,\cdot)b)$. It is clearly
$\lambda$--measurable and
\begin{align}
\nonumber\int_{0}^{1}\abs{\lambda(f_{\xi_{1},\xi_{2}}(t,\cdot)b)}^{2}d\lambda(t)&=\int_{0}^{1}\abs{\int_{0}^{1}f_{\xi_{1},\xi_{2}}(t,s)b(s)d\lambda(s)}^{2}d\lambda(t)\\
\nonumber&\leq\norm{b}^{2}\int_{0}^{1}\left(\int_{0}^{1}\abs{f_{\xi_{1},\xi_{2}}(t,s)}d\lambda(s)\right)^{2}d\lambda(t)\\
\nonumber&\leq\norm{b}^{2}\int_{0}^{1}\int_{0}^{1}\abs{f_{\xi_{1},\xi_{2}}(t,s)}^{2}d\lambda(t)d\lambda(s)<\infty.
\end{align}
Therefore, from equation \eqref{norm_cal} we get,
\begin{align}
\nonumber\underset{a\in C[0,1], \norm{a}_{2}\leq 1}\sup\abs{\int_{0}^{1}a(t)\mathbb{E}_{A}(\xi_{1} b\xi_{2}^{*})(t)d\lambda(t)}&=\underset{a\in C[0,1], \norm{a}_{2}\leq 1}\sup\abs{\int_{0}^{1}a(t)\lambda(f_{\xi_{1},\xi_{2}}(t,\cdot)b)d\lambda(t)}\\
\nonumber
&=\left(\int_{0}^{1}\abs{\lambda(f_{\xi_{1},\xi_{2}}(t,\cdot)b)}^{2}d\lambda(t)\right)^{\frac{1}{2}}<\infty.
\end{align}
\noindent Thus, it follows that $\mathbb{E}_{A}(\xi_{1}
b\xi_{2}^{*})\in L^{2}(A)$ and
\begin{align}
\nonumber \norm{\mathbb{E}_{A}(\xi_{1}
b\xi_{2}^{*})}_{2}^{2}=\int_{0}^{1}\abs{\lambda(f_{\xi_{1},\xi_{2}}(t,\cdot)b)}^{2}d\lambda(t).
\end{align}
\indent Consequently, from Lemma  \ref{identify_disintegrated_measure} we have
\begin{align}
\nonumber \sum_{k\in \mathbb{Z}}\norm{\mathbb{E}_{A}(\xi_{1}v^{k}\xi_{2}^{*})}_{2}^{2}&=\sum_{k\in \mathbb{Z}}\int_{0}^{1}\abs{\eta_{\xi_{1},\xi_{2}}^{t}(1\otimes v^{k})}^{2}d\lambda(t)=\int_{0}^{1}\sum_{k\in \mathbb{Z}}\abs{\eta_{\xi_{1},\xi_{2}}^{t}(1\otimes v^{k})}^{2}d\lambda(t)\\
\nonumber &=\int_{0}^{1}\sum_{k\in \mathbb{Z}}\abs{\int_{0}^{1}f_{\xi_{1},\xi_{2}}(t,s)v^{k}(s)d\lambda(s)}^{2}d\lambda(t) \text{ }\text{ }(\text{Lemma 3.6 \cite{Muk1}})\\
\nonumber&=\int_{0}^{1}\norm{f_{\xi_{1},\xi_{2}}(t,\cdot)}_{L^{2}(\lambda)}^{2}d\lambda(t)\\
\nonumber&=\int_{[0,1]\times
[0,1]}\abs{f_{\xi_{1},\xi_{2}}(t,s)}^{2}d(\lambda\otimes
\lambda)(t,s)<\infty.
\end{align}
Finally, let $S^{\prime}=\text{ span }S$.
\end{proof}

\begin{Remark}\label{Rem:append}
\emph{When $A$ is a masa of product class, the set $S$ in Lemma \ref{Lemma:appendix} 
can be chosen such that $\frac{d\eta_{\zeta}}{d(\lambda\otimes \lambda)}$ is essentially 
bounded for $\zeta\in S$ $($see proof of Theorem 2.5 and 2.7, \cite{Muk2}$)$. So the same is true 
for vectors in $S^{\prime}=\text{span }S$.}
\end{Remark}

\end{document}